\documentclass[a4paper,11pt]{article}

\usepackage{latexsym}
\usepackage{amsthm}
\usepackage{a4wide}

\usepackage{pgf}
\usepackage{graphicx}
\usepackage{subfigure}
\usepackage{epstopdf}

\title{On the structure of isentropes of real polynomials}
\author{O. Kozlovski}
\date{}

\newtheorem{mytheorem}{Theorem}

\newtheorem*{corollary}{Corollary}

\newtheorem{lemma}{Lemma}[section]
\newtheorem*{fact}{Fact}
\newtheorem{question}{Question}
\newtheorem*{conjecture}{Conjecture}

\newtheorem{definition}{Definition}

\newenvironment{theorem-proof}[1]
{\underline{Proof of Theorem~\ref{#1}} \newline}
{\hfill {$\Box$}}

\newfont{\Bb}{msbm10}
\newcommand{\N}{\mbox{\Bb N}}

\newcommand{\R}{\mbox{\Bb R}}

\newcommand{\Image}{\mathop{\mathrm{Image}}}
\newcommand{\Dom}{\mathop{\mathrm{Dom}}}
\newcommand{\RDom}{\mathop{\mathrm{RDom}}}

\newcommand{\pmin}{\mathop{\mathrm{P}_\mathrm{min}}}
\newcommand{\mod}{\mathop{\mathrm{mod}}}
\newcommand{\htop}{\mathop{\mathrm{h}_{\mathrm{top}}}}
\newcommand{\sign}{s}
\newcommand{\bnd}{\mathsf b}
\newcommand{\ppts}{\mathsf P}

\newcommand{\tLambda}{\tilde \Lambda}
\newcommand{\tGamma}{\tilde \Gamma}

\newcommand{\Ec}{{\cal E}}
\newcommand{\Ic}{{\cal X}}
\newcommand{\Pl}{{\cal Q}}
\newcommand{\It}{{\cal I}}
\newcommand{\Lc}{{\cal L}}
\newcommand{\ph}{{\cal PH}}
\newcommand{\sh}{{\cal SH}}
\newcommand{\shp}{{\cal SH}^{\cal P}}
\newcommand{\bLc}{\Lc}
\newcommand{\tLc}{{\tilde \Lc}}
\newcommand{\teLc}{{\tilde \Lc}^\epsilon}
\newcommand{\btLc}{\tilde \Lc}

\newcommand{\eLc}{\Lc^\epsilon}

\newcommand{\Mc}{{\cal M}}
\newcommand{\bMc}{\Mc}
\newcommand{\Pc}{{\cal P}}
\newcommand{\bPc}{\Pc}
\newcommand{\oPc}{\Pc^0}
\newcommand{\ePc}{\Pc^\epsilon}
\newcommand{\pPc}{\Pc^+}
\newcommand{\tPc}{{\tilde \Pc}}

\newcommand{\btPc}{\tPc}

\begin{document}

\maketitle

\begin{abstract}
  In this paper we will modify the Milnor--Thurston map, which maps a
  one dimensional mapping to a piece-wise linear of the same entropy,
  and study its properties. This will allow us to give a simple proof
  of monotonicity of topological entropy for real polynomials and
  better understand when a one dimensional map can and cannot be
  approximated by hyperbolic maps of the same entropy. In particular,
  we will find maps of particular combinatorics which cannot be
  approximated by hyperbolic maps of the same entropy.
\end{abstract}

\section{Introduction}
\label{sec:introduction}

In this paper we will study sets of one dimensional real polynomial
maps which have the same topological entropy which we will be calling
{\it isentropes}.

First the structure of isentropes was understood for the family of
quadratic maps: every isentrope in this case is connected and, since
the parameter space of the normalised quadratic maps is one
dimensional, is either a point or an interval, see \cite{M-Th-88},
\cite{D-H-84}, \cite{D-95}, \cite{Tsujii00}. When the dimension of the
parameter space increases, the structure of isentropes becomes much
more complicated. Even establishing whether isentropes are connected
for families of real polynomials with all critical points real took
quite an effort: in 1992 Milnor conjectured that isentropes are
connected in this case and proved it with Tresser for polynomials of
degree three (\cite{Mil-Tresser-00}); the general case was proved
later by Bruin and van Strien in \cite{BvS}.  However, it is still
unknown if isentropes are connected for real polynomial maps when one
allows some critical points to be complex (though we have made some
progress in this direction and we can prove the connectedness of
isentropes for some families (e.g. $x \mapsto x^4+ax^2+b$) where
complex critical points are allowed, see
Section~\ref{sec:polynomial-model}).

The main goal of this paper is to develop a set of tools which gives a
better understanding of the structure of isentropes and is used to
prove monotonicity. The strategy is based on some modifications of the
Milnor-Thurston map which maps every one dimensional smooth map to a
piece-wise linear map with constant slopes of the same entropy. We
will demonstrate how it works on two problems: we will generalise and
give a much simpler proof of monotonicity of topological entropy
(i.e. we will reprove the main results of \cite{BvS} in a more general
setting), and then we will make some progress in answering one of
Thurston's questions, see below.

The proof in \cite{BvS} is rather complicated and long. Let us review
some general ideas used to prove monotonicity of entropy.

We start with defining what we mean by a monotone map.
\begin{definition}
  Let $X$ and $Y$ be some topological spaces and $F:X \to Y$ be a
  map. We say that the map $F$ is \emph{monotone} if for any
  $y\in Y$ the set $F^{-1}(y)$ is connected.
\end{definition}
The following simple fact will be proved in the Appendix:
\begin{fact}
  Let $X$ and $Y$ be compact connected topological spaces, and the map
  $F:X\to Y$ be continuous, surjective and monotone. Let $Z\subset Y$
  be a connected subset of $Y$. Then the preimage of $Z$ under $F$ is
  connected.
\end{fact}

This statement enables us to use the following strategy for proving
 monotonicity of entropy. Let $X$ be a connected
component of the space of polynomials of given degree with real
critical points (actually, this approach would work for any space of
maps). Now suppose we can find another space of maps $Y$ which is
somewhat ``simpler'' than $X$ and has the following properties:
\begin{itemize}
\item There is a map $F:X \to Y$ which is continuous, surjective and monotone.
\item The map $F$ preserves the topological entropy.
\item The map $\htop : Y \to \R$ is monotone.
\end{itemize}
Then, due to Fact above the map $\htop: X \to \R$ which can be
seen as the composition $\htop|_Y \circ F$ is monotone.

In \cite{BvS} the authors use the space of stunted sawtooth maps as
the probe space $Y$. Stunted sawtooth maps were introduced in
\cite{Mil-Tresser-00}. They are piece-wise linear maps whose branches have slopes
$\pm \mbox{constant}$ or 0. It is rather easy to show monotonicity of
$\htop|_Y$. The map $F$ is defined using the kneading invariants of
the maps and, thus, the maps $f\in X$ and $F(f)\in Y$ have the same
combinatorial structure. This immediately implies that $F$ preserves
the topological entropy. To prove monotonicity of $F$ one should
use the rigidity result for real polynomials, see \cite{KSS1},
\cite{KSS2}, \cite{Clark2017}.

So far the strategy worked out perfectly, but now some problems
arise. It turns out that the map $F$ is neither continuous nor
surjective. The authors of \cite{BvS} had to overcome the lack of
these two properties which was not straightforward.

Now let us try a different probe space $Y$, for example the usual
space of piece-wise linear maps with constant slopes.  The map $F$ in
this case is given by the Milnor-Thurston map
\cite{M-Th-88}. However, again the map $F$ is not continuous and
not surjective.

In this paper we use a slight modification of the space of piece-wise
linear maps of constant slopes. This modification makes the
Milnor-Thurston map continuous and surjective and all other required
properties we get almost for free.

Another new ingredient we introduce is the notion of multi-interval
maps. At first sight one might think that these maps should not be of
great use: after all, the dynamics of a multi-interval map can be
described in terms of a usual one dimensional interval map. However,
such multi-interval maps provide a useful decomposition of iterates of
a map and will enable us to formulate certain results in the more
general (and useful) settings.

\vspace{2mm}

As we have already mentioned one of the aims of this paper is to give
a short proof of monotonicity of topological entropy. There is
another profound reason for finding different approaches to this
problem. The stunted sawtooth maps used in \cite{Mil-Tresser-00} and
\cite{BvS} have rather complicated dynamics and though it is easy to
prove that in the space of stunted sawtooth maps sets of constant
topological entropy are connected, the structure of the isentropes is
completely unclear and it is impossible to see what
stunted sawtooth maps belong to a given isentrope.

On the other hand, in the space of piece-wise linear maps of constant
slopes the isentropes can be easily understood: such an isentrope
consists of maps whose slopes are equal to $\pm\exp(h)$ where $h$ is
the topological entropy of the given isentrope.

The following question was asked by W.~Thurston:
\begin{question}
  Consider the space of real polynomials of degree $d>2$ with all
  critical points real. Does there exist a dense set
  $H\subset [0,\log(d)]$ of entropy levels such that the hyperbolic
  polynomials are dense in the isentrope of entropy $h$ for every
  $h\in H$?
\end{question}
As usual we call a polynomial hyperbolic if the iterates of all
critical points converge to attracting periodic points and there are
no neutral periodic points. It is clear that there are only countably
many combinatorially different hyperbolic maps, so there exists at
most countably many entropy levels whose isentropes contain hyperbolic
maps. In fact, a simple argument (presented in Section~\ref{sec:tq})
will show that the entropy of a hyperbolic map is always the logarithm
of an algebraic number. In view of this discussion one might ask
questions related to Thurston's one:
\begin{question} \label{q:2}
  Consider the space of real polynomials of degree $d>2$ with all
  critical points real. Do there exist isentropes of positive entropy
  which contain hyperbolic maps of infinitely many different
  combinatorial types? Is there a dense set of entropy levels with
  such the property?
\end{question}
Of course, an affirmative answer on Thurston's question implies the
affirmative answer of the above questions, however we conjecture that
the answer on Thurston's question is negative. More precisely we
conjecture the following:
\begin{conjecture}
  In the space of polynomials of degree $d > 2$ with all critical
  points real there are no isentropes of entropy $h \in (0, \log d)$ where
  hyperbolic polynomials are dense.
\end{conjecture}
The results of this paper give some insight on how one can prove the
conjecture. In Section~\ref{sec:tq} we will explain how to
reduce this conjecture first to a question about piece-wise linear
maps and then to some number theory question. In fact we will find a
combinatorial obstruction which prevents a map from being approximable
by a hyperbolic map of the same topological entropy. Also, we will
demonstrate that in case of cubic polynomials the answer on the first
part of question~\ref{q:2} is positive.

The paper is structured as follows. After introducing some necessary
notation we state monotonicity of entropy theorems in
Section~\ref{sec:polynomial-model}. Then we introduce the space of
piece-wise linear maps with constant slopes, define the
Milnor-Thurston map and prove that after an appropriate modification
this map becomes continuous. This will take
Sections~\ref{sec:linear-model}--\ref{sec:continuity-lambda}. The
proof of the monotonicity theorems are in
Sections~\ref{sec:proof-main-result} and
\ref{sec:case-non-degenerate}. Then we will study when a map cannot be
approximated by hyperbolic maps of the same entropy and discover that
under certain (rather non-restrictive) condition a map which has all
critical points in basins of periodic attractors except one critical
point, cannot be approximated by hyperbolic maps of the same entropy
(Section~\ref{sec:tq}). Finally, we study more the mentioned
condition, give some examples when it is not satisfied, prove that it
is always satisfied if the entropy of the map is larger than $\log 3$
and argue that every isentrope should have such a ``codimension one
hyperbolic'' map (Sections~\ref{sec:except-isentr} and
\ref{sec:ordinary-isentropes}).

There are many more other open questions related to monotonicity
of entropy where the approach introduced here can be
useful. For example, it is unknown if the isentropes in the space of
real polynomials are contractible.  We suggest the reader to consult
\cite{vS-14} and the introduction of \cite{BvS} where the history and
importance of monotonicity of entropy together with remaining
open problems are discussed with very fine details.

\section{Multi-interval Multi-modal maps}
\label{sec:multi-interval-multi}

Surprisingly enough to prove monotonicity of entropy for polynomials
we will have to consider more general spaces of maps which we will
call multi-interval multi-modal, and which are introduced in this
section. Because of use of these multi-interval maps our main theorems
will apply to the wider class of spaces compared to \cite{BvS},
however our way of proof will require these maps even for the proof of
the monotonicity of the entropy just for the space of polynomial maps
considered in \cite{BvS}.

Let $I=\cup_{k=1}^N I_k$ be a union of disjoint intervals and $f: I
\to I$ be a differentiable map which maps the set of boundary
points of $I$ to itself. We will call such a map \emph{multi-interval
  multi-modal}. The domain of definition $I$ of $f$ will be denoted by
$\Dom(f)$.

Every interval $I_k$ is mapped by $f$ into another interval which we
denote $I_{\sigma(k)}$ where $\sigma : \{1,\ldots,N\} \to
\{1,\ldots,N\}$. So $f(I_k) \subset I_{\sigma(k)}$ and $f(\partial
I_k) \subset \partial I_{\sigma(k)}$.
Note that we do not assume that $\sigma$ is a permutation.

We also define two more function associated to $f$: $l(k)$ will denote
the number of critical points of the map $f|_{I_k}$ counting with
their multiplicities; $\sign(k)$ is
defined to be $+1$ if $f$ is mapping the left boundary point of $I_k$
onto the left boundary point of $I_\sigma(k)$ and $-1$ otherwise. The
total number of critical points will be denoted by $|l|:= \sum_{k=1}^N
l(k)$.

The space of $C^b$ multi-interval multi-modal maps has the topology
induced by the $C^b$ norm.  $\Mc_{N,\sigma,l,\sign}^b$ will denote all
multi-interval multi-modal $C^b$  maps with the prescribed combinatorial data
$N$, $\sigma$, $l$, and $\sign$. Notice that
$\Mc_{N,\sigma,l,\sign}^b$ is a connected set.

We will need to consider subsets of $\Mc_{N,\sigma,l,\sign}^b$ defined
as follows. The set of boundary points of intervals $I_k$ is mapped to
itself by $f$, and this map depends only on the combinatorial
information $N,\sigma,l,\sign$. Let $\ppts$ be the set of orbits of
periodic boundary points and $\bnd : \ppts \to \{0,1\}$ be a function
which assumes only two values $\{0,1\}$.  Then
$\Mc_{N,\sigma,l,\sign, \bnd}^b$ defined as a subset of
$\Mc_{N,\sigma,l,\sign}^b$ such that $|Df^n(x)|\ge \bnd(p)$ if $x$ is
a periodic boundary point of period $n$ and $p$ is the periodic orbit
corresponding to $x$. Here $Df$ denotes the derivative of $f$.  Clearly, if $\bnd$ is a zero function, then
$\Mc_{N,\sigma,l,\sign}^b = \Mc_{N,\sigma,l,\sign, \bnd}^b$.

To simplify notation we set $\Ic=\{N,\sigma,l,\sign,\bnd \}$ and will
write $\Mc_\Ic^b$ instead of $\Mc_{N,\sigma,l,\sign, \bnd}^b$ when it
does not create a confusion.  We will not distinguish maps in
$\bMc_\Ic^b$ which can be obtained from each other by a linear
rescaling of intervals $I_k$, so we can assume that all intervals
$I_k$ are of the unit length. Also, notice that for any combinatorial
information $\Ic$ the space $\Mc_\Ic$ is connected.

If $\sigma$ is a cyclic permutation, we will call the space
$\bMc_\Ic^b$ \emph{cyclic}. If there exists $k_0$ such that for any $k
\in \{1,\ldots,N\}$ there is $n\in \N$ such that $\sigma^n(k)=k_0$,
then the corresponding space $\bMc_\Ic^b$ is called
\emph{primitive}. For every primitive space $\bMc_\Ic^b$ there exists
a unique set of subintervals $I_{k_1},\ldots,I_{k_{N'}}$ such that the
restriction of the maps in $\bMc_\Ic^b$ to the union of these
subintervals forms a cyclic space $\bMc_{\Ic'}^b$ with an appropriate
$\Ic'$. The number $N'$ will be called the \emph{period} of the
primitive space $\bMc_\Ic^b$.

Every space $\bMc_\Ic^b$ can be decomposed into a Cartesian product of
primitive spaces:
$$
\bMc_\Ic^b=\bMc_{\Ic_1}^b \times \cdots \times \bMc_{\Ic_m}^b,
$$
where all $\bMc_{\Ic_i}^b$ are primitive. The minimum of all periods
of $\bMc_{\Ic_i}^b$ will be called the \emph{minimal period} of
$\bMc_{\Ic}^b$ and will be denoted by $\pmin$.

Given two data sets $\Ic$ and $\Ic'$ we will say that $\Ic'$ is
\emph{subordinate} to $\Ic$ if either $|l|>|l'|$ or $|l|=|l'|$ and
$\pmin(\Ic)<\pmin(\Ic')$.

Finally, if $l(k)>0$ for any $k \le N$ such that $k \not \in
\Image(\sigma)$, the corresponding space $\bMc_\Ic^b$ will be called
\emph{essential}. In other words, an essential space cannot have an
interval without critical points which does not contain an image of another interval. 

The multi-interval multi-modal maps are not much different from just the
multi-modal maps and the combinatorial theory of one dimensional maps
can be applied to them.

The \emph{basin of attraction} of a non-repelling periodic point $x$
of the map $f$ is defined as the interior of all points whose
trajectories converge to the orbit of $x$ under iterates of $f$ and
denoted by $B(f,x)$. The intervals of $B(f,x)$ which contain points of
orbit of $x$ is called the \emph{immediate basin of
  attraction}. Basins of attraction whose immediate basins of
attraction contain critical points of $f$ are called \emph{essential}.
Finally, the \emph{basin of attraction} of the map $f$ is the union of
basins of attraction of all non-repelling periodic points of $f$ and
denoted by $B(f)$.

Two maps $f_1,f_2 \in \bMc_\Ic^1$ are called \emph{semi-conjugate} if
there exists a continuous monotone map $H:I\to I$ such that
$H(I_k)=I_k$ for all $k$, the map $H$ maps the critical points of $f_1$
onto the critical points of $f_2$ of the same order and $H\circ f_1 = f_2 \circ H$.

Two maps $f_1,f_2 \in \bMc_\Ic^1$ are called \emph{partially
  conjugate} if there exists a homeomorphism $H:I\to I$ such that
$H(I_k)=I_k$ for all $k$, the map $H$ maps the critical points of
$f_1$ onto the critical points of $f_2$ of the same order, $H$ maps
the basins of attraction $B(f_1)$ onto the basins of
attraction $B(f_2)$, i.e. $H(B(f_1))=B(f_2)$, and
$H\circ f_1|_{I\setminus B(f_1)} = f_2 \circ H|_{I\setminus B(f_1)}$.

\section{Polynomial model}
\label{sec:polynomial-model}

In the space $\Mc_\Ic^\infty$ consider maps $p$ such that the
restriction of $p$ to any interval $I_k$ is a polynomial of degree
$l(k)+1$. Notice that this implies that all critical points of the
polynomial $p|_{I_k}$ belong to the interval $I_k$ and, therefore,
$p_{I_k}$ has non-positive Schwarzian derivative. We denote the space
of such maps by $\Pc_\Ic$.

To state the main result of this paper we will use the following
notation: if $X$ is a space of maps (e.g. $\bMc_\Ic^1$ or $\bPc_\Ic$),
then for any $h\ge 0$ we define
\begin{eqnarray*}
  X(=h)&:=&\{f\in X:\, \htop(f)=h\}\\
  X(\le h)&:=&\{f\in X:\, \htop(f) \le h\}.
\end{eqnarray*}

\begin{mytheorem}\label{th:A}
  The isentrope $\bPc_\Ic(=h)$ is connected for any $\Ic$ and $h\ge 0$, in
  other words the map $\htop|_{\bPc_\Ic}$ is monotone.
\end{mytheorem}

Notice that the space $\bMc^b_\Ic$ as well as $\bPc_\Ic$ contains maps
with degenerate critical points. Let us remove these maps and denote by
$\Mc^{0,b}_\Ic \subset \bMc^b_\Ic$ the set of maps which have only
quadratic critical points, and set $\oPc_\Ic:= \bPc_\Ic \cap
\Mc^{0,b}_\Ic$. The topological entropy function is also monotone on
this space:
\begin{mytheorem}\label{th:B}
  The isentrope $\oPc_\Ic(=h)$ is connected for any $\Ic$ and $h\ge 0$.
\end{mytheorem}

We will see that the sets $\Pc_\Ic(\le h)$ and  $\oPc_\Ic(\le h)$ are
connected as well.

Interestingly enough the use of multi-interval spaces enable us to
prove the connectedness of isentropes for some families. For example,
the family $p_4: x \mapsto x^4+ax^2+b$ can be seen as a composition of
two quadratic maps: $p_4(x) = (x^2+\frac 12 a)^2 + b-\frac 14
a^2$. Then the Theorem~\ref{th:A} applied in the case
$\Ic=\{2,(1\to 2\to 1), (1,1),(-1,-1)\}$ implies that the isentropes
in the family $p_4$ are connected. Notice that for some values of
parameters $(a,b)$ the map $p_4$ is a real unicritical map of positive
entropy having complex critical points. In general, the following
corollary holds. Let $\Pl_{d,s}$ denote the set of all real
polynomials of degree $d$ which satisfy the following conditions:
all critical points of these polynomials are real and in the
unit interval; such polynomials define proper maps of the unit
interval into itself; the leading coefficients of the polynomials have
the same sign $s$.

\begin{corollary}
  Consider a family of real polynomial maps obtained as a composition
  of polynomials $p_n\circ \cdots\circ p_1$, where
  $p_1\in \Pl_{d_1,s_1}, \ldots,p_n \in \Pl_{d_n,s_n}$ for some
  $s_1,\ldots,s_n\in \{+,-\}$ and natural $d_1,\ldots,d_n$. Then the
  isentropes in such a family are connected. Moreover, if maps with
  degenerate critical points are removed from the family, the
  isentropes remain connected.
\end{corollary}

\vspace{5mm}

One of the ingredients of the proof is based on the Rigidity Theorems
\cite{KSS1}, \cite{Clark2017}  and can be proved for multi-interval maps
exactly in the same way as Lemma 3.12 in \cite{BvS}. Later this lemma will
enable us to prove monotonicity of a certain map.

\begin{lemma}\label{lm:connect1}
  Let $f$ be in $\bPc_\Ic$ and let $\ph_\Ic(f) \subset \bPc_\Ic$ denote the
  set of maps partially conjugate to $f$. Then the set $\ph_\Ic(f)$ is
  connected.
\end{lemma}

$\tPc_\Ic$ will denote the quotient space of $\Pc_\Ic$ with respect to
the partial conjugacy. For any map $f\in \Mc_\Ic^1$ there
exists a map $p\in \Pc_\Ic$ which is semi-conjugate to $f$. Moreover,
this semi-conjugacy collapses only intervals which are in the non-essential basins of
attraction and wandering intervals, see
Theorem 6.4, page 156 in \cite{Melo1995}. If there
are two maps $p_1,p_2 \in \Pc_\Ic$ which are both semi-conjugate to
$f$, then $p_1$ and $p_2$ are partially conjugate. Thus we can
define the map $\Upsilon : \Mc_\Ic^1 \to \tPc_\Ic$ so $\Upsilon(f)$ is
a set of partially conjugate polynomial maps which contains a map
semi-conjugate to $f$. Obviously, $\Upsilon$ is surjective, it is also
easy to see that it is continuous.

\section{Piece-Wise Linear model}
\label{sec:linear-model}

Fix $h\ge 0$ and let us consider a space of piece-wise linear maps whose
slopes are $\pm e^h$ and which satisfy the same combinatorial
properties as $\Mc_\Ic$. More precisely, for
$\Ic=\{N,\sigma,l,\sign,\bnd\}$ as before we will study the space of
piece-wise linear maps $q: I \to I$, where $I=\cup_{k=1}^N I_k$, $q$
maps boundary of $I$ to itself, for any $k\le N$ one has $q(I_k)\subset
I_{\sigma(k)}$, there are precisely $l(k)$ turning points of $q$ in
the interval $I_k$ (though some of them we allow to collide), and $s(k)$ tells
us if $q$ is decreasing or increasing at the left boundary point of the
interval $I_k$. The function $\bnd$ does not play any role here.

To normalise the settings and slightly abusing the notation we
consider the points $a_0=0\le a_1 \le \cdots \le a_N=1$ and set
$I_k=[a_{k-1},a_k]$. Then the map $q$ is discontinuous at points
$a_k$. To distinguish the different values of the map $q$ on different
sides of the points $a_k$ we introduce the following notation:
$q(a^+_k) = \lim_{x \searrow a_k} q(x)$ and $q(a^-_k) = \lim_{x \nearrow a_k} q(x)$.

Given a map $q$ described above for any branch of $q$ there exists $b$
such that for that branch we have $q(x) = \pm e^h x + b$. So, any map
as above can be described by the following data: the combinatorial
data $\Ic$, the points $a_k$ for $k=0,\ldots,N$, the coefficients
$b_k^i$ for $k=1,\ldots,N$, $i=0,\ldots,l(k)$. The $i^{th}$ branch of
$q$ on $I_k$ is then given by the formula
$q(x) = (-1)^i s(k) e^h x + b_k^i$.

Of course, not for all possible choices of $a_k$ and $b_k^i$ there is
a map which has this prescribed data. The following conditions should
be satisfied:
\begin{itemize}
\item The $i^{th}$ turning point $c_k^i$ of $q|_{I_k}$ must belong to $I_k$. The value of $c_k^i$ can be found from
  $$
  -(-1)^{i} s(k) e^h c_k^i + b_k^{i-1} = (-1)^{i} s(k) e^h c_k^i +
  b_k^{i},
  $$
  so $c_k^i = \frac 12 (-1)^i s(k) e^{-h} (b_k^{i-1}-b_k^i)$. All
  turning points should be ordered correctly,  thus the following
  inequalities must hold:
  \begin{equation} \label{ieq:1}
    0=a_0\le c_1^1\le c_1^2\le \cdots \le c_1^{l(1)}\le a_1\le
    c_2^1 \le \cdots \le a_N=1
  \end{equation}
\item The turning values should belong to the corresponding interval
  as well. The turning value $q(c_k^i)$ is $\frac 12
  (b_k^{i-1}+b_k^{i})$, therefore
  \begin{equation} \label{ieq:2}
    a_{\sigma(k)-1} \le \frac 12 (b_k^{i-1}+b_k^{i}) \le a_{\sigma(k)}
  \end{equation}
  should be satisfied for all $k=1,\ldots,N$ and $i=1,\ldots,l(k)$.
\item Finally, the map $q$ must have the prescribed values at the
  boundary points of the intervals $I_k$. We know that
  $q(I_k) \subset I_{\sigma(k)}$ and the boundary points of $I_k$ are
  mapped to the boundary points of $I_{\sigma(k)}$. Let
  $q(a_{k-1}^+) = a_{\sigma_l(k)}$ and
  $q(a_{k}^-) = a_{\sigma_r(k)}$, where the functions $\sigma_l$ and
  $\sigma_r$ are completely defined by the combinatorial data $\Ic$
  and $\sigma_l(k)$ and $\sigma_r(k)$ can assume one of the two values :
  $\sigma(k)$ or $\sigma(k)-1$ depending on $s(k)$ and $l(k)$. Therefore,
  \begin{eqnarray}
    \label{eq:4}
    s(k)e^h a_{k-1} + b^0_k & = & a_{\sigma_l(k)},\\
    \label{eq:5}
    (-1)^{l(k)}s(k)e^h a_{k} + b^{l(k)}_k & = & a_{\sigma_r(k)}.
  \end{eqnarray}
\end{itemize}

For given $h$ and $\Ic$ if $a_k$ and $b_k^i$ satisfy the inequalities
and equalities
above, then the corresponding piece-wise linear map described by these
data exists. The set of these maps we will denote by
$\Lc_\Ic(=h)$. Obviously, $\Lc_{\Ic}(=h)$ is a compact subset of $\R^D$
for some $D$ depending on $\Ic$. Moreover, since $\Lc_\Ic(=h)$ is
described by linear inequalities and equalities in $\R^D$, it is connected as an
intersection of finitely many connected convex subsets of $\R^D$. So,
we have proved 
\begin{lemma}
  The set $\Lc_\Ic(=h)$ is connected.
\end{lemma}

Finally, the space $\Lc_\Ic$ we define as  $\Lc_\Ic := \cup_{h> 0}
\Lc_\Ic(=h)$.

Let us repeat that we allow maps in $\Lc_\Ic$ to have colliding
turning points. For example, if two turning points $c_k^i$ and
$c_k^{i+1}$ of the map $q$ collide, i.e. $c_k^i=c_k^{i+1}$, then the
graph of $q$ will have not $|l|+N$ branches as a generic map in
$\Lc_\Ic$ but only $|l|+N -1$ branches and the point $c_k^i$ might not
be a turning point on the graph. However, we will keep track of such
collided points and we will still call them turning. Other
(i.e. non-collided) turning points of $q$ will be called
\emph{simple}.

\section{A link between $\Mc_\Ic$ and $\Lc_\Ic$}
\label{sec:link-between-mc_ic}

Milnor and Thurston \cite{M-Th-88} (see also  \cite{parry-66}) defined the function
$\Lambda : \bMc_\Ic^1 \to \bLc_\Ic$ such that the maps  $f
\in \bMc_\Ic^1$ and $\Lambda(f)$ are semi-conjugate and of the same
topological entropy (they did it for the maps of an interval, but
their construction can be applied to our case with no alterations). 
The particular definition of $\Lambda$ is of no importance for
us, the only thing we are going to use is the fact that for any
function $f$ there exists $q\in \bLc_\Ic$ semi-conjugate to $f$ and of
the same topological entropy.

For a map $q \in \bLc_\Ic$ let us define a set of all maps in
$\bMc_\Ic^b$ which are semi-conjugate to $q$ and denote it by
$\sh^b_\Ic(q)$. Notice that maps in $\sh^b_\Ic(q)$ can have topological entropy
different from the entropy of $q$, and $\htop(q)  \le \htop(f)$ for
any $f\in \sh^b_\Ic(q)$.

The set $\sh^b_\Ic(q)$ is closely related to the notion of a restrictive
interval. An interval $J\subset I$ is called a $restrictive$ interval
of a map $f\in \bMc_\Ic^1$ if there exists $n\in \N$ such that $f^n(J)
\subset J$ and $f^n(\partial J) \subset \partial J$. A connected
component of a  preimage of a restrictive interval we will also call
a restrictive interval.

Fix maps $q \in \bLc_\Ic$, $f \in \sh^1_\Ic(q)$ and let $H$ be the
semi-conjugacy between $f$ and $q$. Suppose that one of turning points
$c_q$ of the map $q$ is periodic of period $n$. The set $H^{-1}(c_q)$
cannot be just a point. Indeed, if  $H^{-1}(c_q)$ is a point, then it
would be a critical point of $f$ and, therefore,  $c_f := H^{-1}(c_q)$
would be a superattractor. Iterates of all points in a neighbourhood of $c_f$
would converge to the orbit of $c_f$, which is impossible if $H$ is
not locally constant near $c_f$.

Thus $H^{-1}(c_q)$ is an interval, and let us define
$J_k:=H^{-1}(q^k(c_q))$ for $k=0,\ldots,n-1$. It is easy to see that
$J_k$ are restrictive intervals, $f(J_k)\subset J_{k+1 (\mod n)}$,
$f(\partial J_k)\subset \partial J_{k+1 (\mod n)}$.  The map $f$
restricted to $\cup_{k=0}^{n-1} J_k$ belongs to the cyclic space
$\bMc_{n,\sigma',l',\sign',\bnd'}$, where $\sigma'$, $l'$, $\sign'$ are
defined in an obvious way. The definition of the function $\bnd'$ is
more subtle and is done as follows. One or both boundary points of
$J_0$ are periodic. Let $x$ be a periodic boundary point of $J_0$ of
period $n_0$ (where $n_0$ is either $n$ or $2n$). If $x$ is an
interior point of $\Dom(f)$, then $x$ cannot be a hyperbolic
attractor. Indeed, otherwise it would attract trajectories of 
points on both sides of $x$, so $H$ must be locally constant around
$x$ and then $x$ cannot be a boundary point of $J_0$. Thus,
$|Df^{n_0}(x)| \ge 1$ and, in this case, we set $\bnd'(x)=1$. If the
point $x$ is a boundary point of $\Dom(f)$, then we set
$\bnd'(x)=\bnd(x)$.

If there exists another turning point $c^2_q$ of $q$ so that
$q^m(c^2_q)=c_q$ where $m$ is minimal with this property, we can do a
similar construction: define $J^2_k := H^{-1}(q^k(c^2_q))$ for
$k=0,\ldots,m-1$. Then again the map $f$ restricted to
$\left(\cup_{k=1}^{n-1} J_k\right) \bigcup \left(\cup_{k=1}^{m-1}
  J^2_k\right)$ is a essential multi-interval multi-modal map.

We can repeat this construction for all periodic turning points of $q$
and for all turning points of $q$ one of whose iterates is mapped onto
a periodic turning point. In this way to any map $f\in \sh^1_\Ic(q)$
we will associate another multi-interval multi-modal map (which is a
restriction of $f$ to the union of the restrictive intervals as above)
in $\bMc_{\Ic_q}^1$ for an appropriate
$\Ic_q:=\{N_q,\sigma_q,l_q,\sign_q,\bnd_q\}$. Notice that $\Ic_q$
depends only on $q$ and is independent of $f$. Also, from the
construction it follows that the space $\bMc_{\Ic_q}^b$ is essential.

The union of all restrictive intervals used in this construction we
will denote by $\RDom(f,q)$. From the definition of $\Lambda$ it follows
that if $q=\Lambda(f)$, then $\htop(f|_{\Dom(f)\setminus\RDom(f,q)})=\htop(f)$ and
$\htop(f|_{\RDom(f,q)})\le\htop(f)$.

Thus for any map $q \in \bLc_\Ic$ there exists a map from
$\sh^1_\Ic(q)$ to $\bMc_{\Ic_q}^1$ defined as above. Notice that
because of the way we have constructed the function $\bnd_q$ this map
is surjective. We will be more
interested in the restriction of this map to the space $\bPc_\Ic$ and
denote this map by $\Gamma_q : \shp_\Ic(q) \to \bMc_{\Ic_q}^\infty$,
where $\shp_\Ic(q)$ denotes $\sh^1_\Ic(q) \cap \bPc_\Ic $. If $q$ does
not have periodic turning points, we set $N_q=0$ and the map
$\Gamma_q$ is trivial.

Let us list a few properties of $\Ic_q$ and $\Gamma_q$. In what
follows we denote the map $\Upsilon \circ \Gamma_q$ by $\tGamma_q$.

\begin{lemma} \label{lm:gamma1}
  For any $q\in \bLc_\Ic$
  \begin{enumerate}
  \item $\Gamma_q$ and $\tGamma_q$ are continuous;
  \item the map $\tGamma_q : \shp_\Ic(q) \to \btPc_{\Ic_q}$ is
    surjective;\label{item:gamma1-2}
  \item if for $p_1,p_2 \in \shp _\Ic(q)$ one has $
    \tGamma_q(p_1) = \tGamma_q(p_2)$, then
    $\Upsilon(p_1)=\Upsilon(p_2)$ (i.e. $p_1$ and $p_2$ are partially conjugate);\label{item:gamma1-3}
  \item the map $\tGamma_q : \shp_\Ic(q) \to \btPc_{\Ic_q}$ is
    monotone, i.e. for any $\upsilon \in \btPc_{\Ic_q}$ the set
    $\tGamma_q^{-1}(\upsilon) $ is
    connected;\label{item:gamma1-4}
  \item if $\Ic$ is cyclic and $\htop(q)>0$, then $\Ic_q$ is subordinate to
    $\Ic$. \label{item:gamma1-5}
  \end{enumerate}
\end{lemma}

\begin{proof}
  The continuity of $\Gamma_q$ is obvious and the map $\tGamma_q$ is a
  composition of two continuous maps.

  The surjectivity is also easy to see: fix any $f\in \sh^1_\Ic(q)$ and
  $\upsilon \in \btPc_{\Ic_q}$, and take $g \in
  \Upsilon^{-1}(\upsilon)$ which has matching derivatives as $f$ at
  boundary points of $I_q$. Then one can glue $g$ into corresponding
  restrictive intervals of $f$ and obtain a map which is still
  semi-conjugate to $q$ and has a prescribed image under $\tGamma_q$.
  Then take $p\in \bPc_\Ic$ semi-conjugate to $f$ given by
  aforementioned Theorem 6.4, \cite{Melo1995}. It is easy to see
  that $\tGamma_q(p) = \upsilon$.

  Claim~\ref{item:gamma1-3} is straightforward: in the set $\Dom(p_i)
  \setminus \RDom(p_i,q)$ the partial conjugacy is
  given by the semi-conjugacies between $p_i$ and $q$, and inside of
  $\RDom(p_i,q)$ it is defined by $\tGamma_q(p_i)$.

  Claim~\ref{item:gamma1-4} follows from Claim~3 and Lemma~\ref{lm:connect1}.

  For the last claim of the lemma consider
  $\Ic_q=\{N_q,\sigma_q,l_q,\sign_q,\bnd_q\}$ and take some $f \in \sh^1_\Ic(q)$.
  Clearly, $|l_q|\le |l|$ and $\pmin(\Ic_q)\ge N$. Suppose that
  $|l_q|=|l|$ and $\pmin(\Ic_q)=N$. In this case each connected
  component of $\Dom(f)$ contains one (and only one) of connected
  components of $\RDom(f,q)$ and all branches of $f|_{\Dom(f)\setminus
    \RDom(f,q)}$ are monotone (as $|l_q|=|l|$). Then
  $\htop(q)\le \htop(f|_{\Dom(f)\setminus \RDom(f,q)})=0$, and we get a
  contradiction. Thus either $|l_q|<|l|$ or $\pmin(\Ic_q)>N$, and,
  therefore, $\Ic_q$ is subordinate to $\Ic$.
\end{proof}

\section{On the continuity of $\Lambda$}
\label{sec:continuity-lambda}

In the previous section we have defined the map $\Lambda : \bMc_\Ic^b
\to \bLc_\Ic$. This map is neither continuous nor surjective. We will
modify the space $\bLc_\Ic$ to fix this.

Two maps $q_1$ and $q_2$ in $\bLc_\Ic$ are called \emph{similar} if
their topological entropies are the same and there exists a map $f \in
\bMc_\Ic^1$ with $\htop(f)=\htop(q_1)=\htop(q_2)$ and which is
semi-conjugate to both $q_1$ and $q_2$. We will denote this by $q_1
\approx q_2$.

For every map $f\in\bMc_\Ic^1$ there exists a map $p \in \bPc_\Ic$
semi-conjugate to $f$ which just collapses the possible wandering
intervals and non-essential basins of attraction. Hence, if
$q_1\approx q_2$, then there exists $p\in \bPc_\Ic$ such that
$\htop(p)=\htop(q_1)=\htop(q_2)$ and $p$ is semi-conjugate to both
$q_1$ and $q_2$.


The relation $\approx$ is reflexive and symmetric, but not necessary
transitive. The relation we are about to introduce will generalise
$\approx$ and will be transitive, thus, it will be an equivalence relation.
Two maps $q_1$ and $q_2$ in $\bLc_\Ic$ are called \emph{related} if
there exist finitely many maps $q_1', \ldots, q_m'\in \bLc_\Ic$ such
that
$$
q_1 \approx q_1'\approx \cdots \approx q_m'\approx q_2.
$$
In this case we will write $q_1 \sim q_2$.

The quotient space of $\bLc_\Ic$ with respect to $\sim$ will be denoted
by $\btLc_\Ic$ and let $\Psi : \bLc_\Ic \to \btLc_\Ic$ be the
corresponding projection. Define $\tLambda := \Psi \circ \Lambda$.

\begin{mytheorem} \label{thr:continuous}
  The map $\tLambda : \bPc_\Ic \to \btLc_\Ic$ is
  surjective and continuous.
\end{mytheorem}

To proof this theorem we need the following lemma first.

\begin{lemma} \label{lm:cont}
  Let $f_i \in \Pc_\Ic$ be a sequence converging to $f_0 \in
  \Pc_\Ic$, the sequence $q_i \in \Lc_\Ic$ converge to $q_0 \in
  \Lc_\Ic$ such that $f_i$ is semi-conjugate to $q_i$ for all $i$.
  Then $f_0$ is semi-conjugate to $q_0$.
\end{lemma}

\begin{proof}
  Let $H_i$ denote the semi-conjugacy between $f_i$ and $q_i$.
  Let us define two function $H_0^-$ and $H_0^+$ by
  \begin{eqnarray*}
    H_0^-(x) &=& \inf_{\{x_i\}:\, x_i \to x} \liminf H_i(x_i),\\
    H_0^+(x) &=& \sup_{\{x_i\}:\, x_i \to x} \limsup H_i(x_i)
  \end{eqnarray*}
  for $x \in I$. In other words, $[H_0^-(x), H_0^+(x)]$
  is the minimal interval containing all limit points of $H_i(x_i)$
  for all sequences $x_i \to x$. From the definition it is clear that $H_0^-(x) \le
  H_0^+(x)$ and since the maps $H_i$ are non-strictly monotone
  increasing, for all $x_1 < x_2$ we have $H_0^+(x_1) \le
  H_0^-(x_2)$. In particular, $H_0^\pm$ are non-strictly monotone
  increasing too.

  It is easy to see that from the definition of $H_0^\pm$ it follows
  that $\liminf H_0^-(x_i) \ge H_0^-(x)$ and $\limsup H_0^+(x_i) \le
  H_0^+(x)$ when $x_i \to x$. Indeed, given $x$ for any $\epsilon >0$
  there exists $\delta >0$ and $N$ such that for all $y \in (x-\delta,
  x+\delta)$ and all $i>N$ one has
  \begin{equation}
    H_i(y)> H_0^-(x) - \epsilon.\label{eq:2}
  \end{equation}
  If this were not true, then there would exist $\epsilon > 0$ and a
  sequences $y_k \to x$ and $i_k \to \infty$ such that
  $H_{i_k}(y_k) \le H_0^-(x) - \epsilon$ and taking the limit we would
  obtain a contradiction with the definition of $H_0^-(x)$. Then
  inequality~(\ref{eq:2}) implies that $H_0^-(y) > H_0^-(x) - \epsilon$
  for all $y\in (x-\delta, x+\delta)$ and we are done.

  Notice that the last property of $H_0^\pm$ implies that if $H_0^-(x)
  = H_0^+(x)$, then the functions $H_0^\pm$ are continuous at $x$.

  Define $A(x):=[H_0^-(x),H_0^+(x)]$. We claim that $q_0(A(x)) \subset
  A(f_0(x))$. Indeed, fix $x_0\in I$ and $y_0\in A(x_0)$ and find $x_i$ such that
  $H_i(x_i)=y_0$ and $x_i \to x_0$. Then, since $H_i$ is a semi-conjugacy
  we have $q_i(y_0) = H_i(f_i(x_i))$. Clearly, $q_i(y_0)$ converges to
  $q(y_0)$, $f_i(x_i)$ converges to $f_0(x_0)$, and the set of limit
  points of the sequence $H_i(f_i(x_i))$ belongs to $A(f_0(x_0))$ because of the
  definition of $H_0^\pm$.

  Suppose that $A(x_0)$ is a non-degenerate interval for some $x_0$.
  The orbit of the interval $A(x_0)$ under the map $q_0$ cannot be
  disconnected because $q_0$ is expanding and has only finitely many
  turning points. So, without loss of generality we can assume that
  $A(x_0)$ contains a turning point and there exists $n>0$ such that
  $q_0^n(A(x_0)) \subset A(x_0)$ . This implies that $x_0$ is a
  periodic critical point of $f_0$ and, therefore, it is a
  superattractor. The corresponding critical points of maps $f_i$,
  where $i$ is sufficiently large will be contained in a basin of
  attraction of a periodic attractor and this basin will contain a
  definite neighbourhood $U$ of $x_0$ which does not depend on $i$
  when $i$ is sufficiently large. Every semi-conjugacy between a $C^1$
  map and piece-wise linear expanding maps must collapse basins of
  attraction, so all maps $H_i$ are constants on $U$ for $i$
  sufficiently large. This implies that $H_0^\pm$ are also constant on
  $U$, so $H_0^-(x_0)=H_0^+(x_0)$ which contradicts the fact that
  $A(x_0)$ is non-degenerate.

  Therefore, we have proved that
  $H_0^-(x)=H_0^+(x)$ for all $x$ and the sequence $H_i$ converges to a
  continuous non-strictly monotone increasing map $H_0$ which is a
  semi-conjugacy between $f_0$ and $q_0$.
\end{proof}

We can proceed with the proof of the theorem now.

\begin{proof}[Proof of Theorem~\ref{thr:continuous}]
  The surjectivity of $\tLambda$ follows from the fact that
  every combinatorics of a piece-wise linear map can be realised by a
  polynomial and the fact that the combinatorially equivalent maps in
  $\bLc_\Ic$ are similar.

  The continuity of $\tLambda$ is a consequence of the above
  lemma. Indeed, take a sequence $f_i \in \Pc_\Ic$ converging to $f_0$
  as in the lemma and let $q_i=\Lambda(f_i)$. Assume $q_i$
  converges to $q_0$.
  From Lemma~\ref{lm:cont} we know that $f_0$ is
  semi-conjugate to $q_0$.  By continuity of the topological entropy
  we know that $\htop(f_0)=\lim \htop(f_i)=\lim
  \htop(q_i)=\htop(q_0)$. The map $\Lambda(f_0)$ is semi-conjugate to
  $f_0$ and has the same topological entropy, hence $q_0 \approx
  \Lambda(f_0)$. Thus, for any sequence $f_i$ converging to $f_0$ we
  have that $\tLambda(f_i)$ converges to $\tLambda(f_0)$.
\end{proof}

Now consider some $f_0\in \Pc_\Ic$ and let $q = \Lambda(f_0)$. As we
already know (by Lemma~\ref{lm:gamma1}(\ref{item:gamma1-3})) for every $\upsilon \in
\tPc_{\Ic_{q}}$ there exists a map $f \in \Pc_\Ic$ which is
semi-conjugate to $q$ and such that $\tGamma_q(f) = \upsilon$. If $\htop(\upsilon) <
\htop(q)$, it is easy to check that in this case $\Lambda(f) = q$. If
$\htop(\upsilon) > \htop(q)$, then $\htop(f) > \htop(q)$ and $\Lambda(f)$
cannot be equal to $q$. The case $\htop(\upsilon) = \htop(q)$ is more
subtle and it is not clear whether $\Lambda(f)$ is $q$ or
not. However, the next lemma shows that $\Lambda(f)$ and $q$ are
similar.

\begin{lemma}\label{lm:4.2}
  For any $q\in \bLc_\Ic$
  $$
  \tGamma_{q}^{-1}(\btPc_{\Ic_q}(\le \htop(q)))
  \,\subset\,
  \bigcup_{q'\approx q}  \Lambda^{-1}(q').
  $$
\end{lemma}

\begin{proof}
  Take $f \in \tGamma_{q}^{-1}(\btPc_{\Ic_q}(\le
  \htop(q)))$. By the definition of $\Gamma_q$ we know that $f$ and $q$
  are semi-conjugate. It is also clear that $\htop(f)=\htop(q)$. Let
  $q'=\Lambda(f)$. Again, by the definition of $\Lambda$, $f$ and $q'$
  are semi-conjugate and have the same topological entropy. Thus,
  $q\approx q'$ and we are done.
\end{proof}

This lemma implies that for any $\tilde q \in \btLc_\Ic$
\begin{eqnarray}
  \label{eq:1}
   \tLambda^{-1}(\tilde q)
   &=&
   \bigcup_{\Psi(q)=\tilde q}
   \tGamma_q^{-1}(\btPc_{\Ic_q}(\le
   \htop(\tilde q))).
\end{eqnarray}
It is straightforward that the left hand side of equality~(\ref{eq:1})
is a subset of the right hand side. The opposite inclusion follows
directly from Lemma~\ref{lm:4.2}.

\section{Proof of the main result (Theorem~\ref{th:A})}
\label{sec:proof-main-result}

In this section we prove that for every $h\ge 0$ the set $\bPc_\Ic(=h)$
is connected. We will do it by induction with respect to the total
number of critical points $|l|$ and $N$.

If $|l|=0$, then the topological entropy of every map in $\bPc_\Ic$ is
zero and we have nothing to do.

Assume that for any $h\ge 0$ the set $\bPc_{\Ic'}(=h)$ is
connected where $\Ic'=\{N',\sigma',l',\sign',\bnd'\}$ with
$|l'|\le L-1$.

Fix some $h_0 >0$. Take $\Ic=\{N,\sigma,l,\sign,\bnd\}$ where $|l|=L$,
$\sigma$ is cyclic, and $N>\log(2) \frac L{h_0}$. In this case every
$f\in \bPc_\Ic$ has topological entropy less than $h_0$ and, therefore,
$\bPc_\Ic(=h)=\emptyset$, $\bPc_\Ic(\le h)=\bPc_\Ic$ are connected
sets for $h \ge h_0$.

Now assume that for any $h\ge h_0$ the set $\bPc_{\Ic'}(=h)$ is
connected where $\Ic'=\{N',\sigma',l',\sign',\bnd'\}$ with either $|l'|\le
L-1$ or $|l'|=L$, $\sigma'$ is cyclic and $N'\ge N+1$.  This is our
induction assumption. At this stage the induction will be done with
respect to $N$ backwards.

The space $\bPc_{\Ic'}$ is connected and the topological entropy
continuously on maps in $\bPc_{\Ic'}$. This implies that since
$\bPc_{\Ic'}(=h)$ is connected, then $\bPc_{\Ic'}(\le h)$ is connected
as well.

If the primary decomposition of the space $\bPc_\Ic$ is
$$
\bPc_\Ic = \bPc_{\Ic_1}\times \cdots \times \bPc_{\Ic_m}
$$
and we know that for any $h\ge h_0$ and $i=1,\ldots,m$ the sets
$\bPc_{\Ic_i}(=h)$ and $\bPc_{\Ic_i}(\le h)$ are connected, then the
sets $\bPc_{\Ic}(=h)$ and $\bPc_{\Ic}(\le h)$ are connected as
well. This implies that we can assume that the sets $\bPc_{\Ic'}(=h)$
and $\bPc_{\Ic'}(\le h)$ are connected if $\Ic'$ is subordinate to
$\Ic=\{N,\sigma,l,\sign,\bnd\}$, where $|l|=L$ and $\sigma$ is cyclic.

Fix cyclic $\Ic=\{N,\sigma,l,\sign,\bnd\}$ and take $\tilde q \in
\btLc_\Ic$ with $\htop(\tilde q)=h\ge h_0$.  Take some $q \in
\Psi^{-1}(\tilde q)$. Due to
Lemma~\ref{lm:gamma1}(\ref{item:gamma1-5}) we know that $\Ic_q$ is
subordinate to $\Ic$, and then from the induction assumption it
follows that $\bPc_{\Ic_q}(\le h)$ is connected. The map $\Upsilon :
\bPc_{\Ic_q} \to \btPc_{\Ic_q}$ is continuous, surjective and
preserves topological entropy, therefore $\btPc_{\Ic_q}(\le
h)=\Upsilon(\bPc_{\Ic_q}(\le h))$ is connected.

The map $\tGamma_q : \shp_\Ic(q) \to \btPc_{\Ic_q}$ is
continuous, surjective and monotone (because of
Lemma~\ref{lm:gamma1}(\ref{item:gamma1-4})), therefore the set
$\tGamma_q^{-1} (\btPc_{\Ic_q}(\le h))$ is
connected due to Lemma~\ref{lm:top}.

Take two similar maps $q_1,q_2 \in \Psi^{-1}(\tilde q)$, so
$q_1\approx q_2$. By definition there exists $p\in \bPc_\Ic$ which has
the same entropy as $q_1$ and $q_2$ and which is semi-conjugate to
$q_1$ and $q_2$. This implies that $p \in \tGamma_{q_i}^{-1}
(\btPc_{\Ic_{q_i}}(\le h))$, where $i=1,2$, and therefore the set
$$
\tGamma_{q_1}^{-1} (\btPc_{\Ic_{q_1}}(\le h))
\,\bigcup\,
\tGamma_{q_2}^{-1} (\btPc_{\Ic_{q_2}}(\le h))
$$
is connected. Using equality~(\ref{eq:1}) we get that the set 
$$
\tLambda^{-1}(\tilde q)
=
\bigcup_{\Psi(q)=\tilde q}
\tGamma_q^{-1}(\btPc_{\Ic_q}(\le h))
$$
is connected as well.

The set $\bLc_\Ic(=h)$ is connected, so is the set $\btLc_\Ic(=h)$. The
map $\tLambda : \bPc_\Ic \to \btLc_\Ic$ is continuous,
surjective and monotone as we just have proved. Thus, due to Lemma~\ref{lm:top}
$\bPc_\Ic(=h) = \tLambda^{-1}(\btLc_\Ic(=h))$ is
connected and we are done.

Finally, notice that the same argument proves that the set $\Pc_\Ic(\le h)$
is connected for any $h>0$.  Then the set
 $$\Pc_\Ic(=0) = \bigcap_{h>0} \Pc_\Ic(\le h)$$
is connected as an intersection of compact connected nested sets.

\section{Case of non-degenerate maps (proof of Theorem~\ref{th:B})}
\label{sec:case-non-degenerate}

In this section we will modify the proof of Theorem~\ref{th:A} given
in the previous section and prove Theorem~\ref{th:B}.

First, for given $\epsilon >0$ let us define the space $\eLc_\Ic
\subset \Lc_\Ic$ as the set of all maps $q \in \Lc_\Ic$ such that
the distance between any turning points and the distance from the
turning points to the
boundaries of $I_k$ are greater or equal than $\epsilon$. This
space is closed and the set $\eLc_\Ic(=h)$ is connected. Indeed, to
describe the set of parameters of $\eLc_\Ic(=h)$ we have to solve
inequalities similar to (\ref{ieq:1}) and (\ref{ieq:2}). More
precisely, the inequality (\ref{ieq:2}) stays the same and
(\ref{ieq:1}) should be replaced by
\begin{equation} \label{ieq:3}
  a_0+\epsilon \le c_1^1\le c_1^1+\epsilon \le c_1^2\le \cdots \le
  c_1^{l(1)-1}+\epsilon \le c_1^{l(1)}\le a_1-\epsilon \le
  a_1+\epsilon \le
  c_2^1 \le \cdots \le a_N-\epsilon
\end{equation}
Again $\eLc_\Ic(=h)$ is described by linear inequalities in $\R^D$, and
it is connected as an intersection of finitely many connected convex
subsets of $\R^D$. 

Next we define the space $\teLc_\Ic$ in the exactly same way as we did
in Section~\ref{sec:continuity-lambda}. More precisely, $\teLc_\Ic$ is
a subset of $\tLc_\Ic$ such that each equivalence class in $\teLc_\Ic$
contains an element of $\eLc_\Ic$. Since $\eLc_\Ic$ is compact, the
space $\teLc_\Ic$ is compact too.  We can also define the space
$\ePc_\Ic \subset \oPc_\Ic$ by setting it to be equal to
$\tLambda^{-1}(\teLc_\Ic)$. Since $\tLambda$ is continuous, $\ePc_\Ic$
is compact.  Using the same prove as in the previous section without
any alterations one can show that the set $\ePc_\Ic(=h)$ is connected.

For any $\epsilon_2 > \epsilon_1 >0$ it is clear that
$\bPc^{\epsilon_2}_\Ic \subset \bPc^{\epsilon_1}_\Ic$ and
$\bPc^{\epsilon_2}_\Ic(=h) \subset \bPc^{\epsilon_1}(=h)$. Since
$\ePc_\Ic(=h)$ are connected we get that the union
$ \cup_{\epsilon>0} \ePc_\Ic(=h)$ is connected as well.

Let us denote $\cup_{\epsilon >0} \ePc_\Ic$ by $\pPc_\Ic$. This set is
a subset of $\oPc_\Ic$ but does not coincide with it. Let us see the
structure of $\oPc_\Ic \setminus \pPc_\Ic$.

There are
polynomials with all critical points non-degenerate, but which are
semi-conjugate to a piece-wise linear maps with collided turning
points.
Take $p_0\in \oPc_\Ic(=h) \setminus \pPc_\Ic$ and let
$q_0 = \Lambda(p_0)$. The map $q_0$ cannot have all its turning points
distinct because otherwise $q_0$  would belong to $\eLc_\Ic$ for
some $\epsilon>0$ and $p_0$ would belong to $\pPc_\Ic$. Hence, $q_0$
must have some collided turning points and it
belongs to the boundary of $\Lc_\Ic$. Moreover, these collided turning
points of $q$ must be periodic, otherwise $p_0$ would have a degenerate
critical point.
Recall that $\shp_\Ic(q_0)$ denotes all polynomials in $\Pc_\Ic$ which
are semi-conjugate to $q_0$. Thus $p_0 \in \shp_\Ic(q_0)\cap
\oPc_\Ic(=h)$ and we have the following decomposition formula for $\oPc_\Ic(h)$:
\begin{equation} \label{eq:op}
\oPc_\Ic(=h) \,=\, \pPc_\Ic(=h)\, \bigcup \,\left(
 \cup_{q\in \Lc_\Ic^\partial(=h)}
 \left( \shp_\Ic(q)\cap \oPc_\Ic(=h) \right)\right),\label{eq:3}
\end{equation}
where $\Lc_\Ic^\partial$ denotes all maps in $\Lc_\Ic$ which have
collided periodic turning points.

We now going to finish the proof that $\oPc_\Ic(=h)$ is connected.
This will be done by induction similar to one in the previous section.
We again fix $h_0>0$ and $\Ic$, and assume that $\oPc_{\Ic'}(=h)$ and
$\oPc_{\Ic'}(\le h)$ are connected for all $h\ge h_0$ and all $\Ic'$
subordinate to $\Ic$.

Fix a map $q_0$ as in a paragraph above, i.e. $q_0 = \Lambda (p_0)$
where $p_0 \in \oPc_\Ic(=h)$. The map $q_0$ has some periodic collided
turning points. Denote one of these turning points by $t$, its period
denote by $n$, and suppose that the number of turning points collided
at $t$ is $k_t$.  To simplify the exposition we will assume that $q_0$
is increasing at $t$, all other turning points of $q_0$ are simple
(i.e. all colliding turning points are concentrated at $t$) and
non-periodic (in particular, the orbit of $t$ does not contain other
turning points). The arguments below are quite general and these
restrictions can be easily dropped.

By the definition we know that a
partial conjugacy maps critical points onto critical points and it
preserves the order of critical points. This implies that if a
polynomial in a given partial
conjugacy class has all critical points non-degenerate, then all other
polynomials from this partial conjugacy class have all critical points
non-degenerate too. Thus, the set $\tPc_{\Ic'}^0 = \Upsilon (\oPc_{\Ic'})$ is
well defined and by the induction assumption the sets
$\tPc_{\Ic'}^0(=h)$ and $\tPc_{\Ic'}^0(\le h)$ are connected for
all $\Ic'$ subordinate to $\Ic$.

We already know that the map $\tGamma_{q_0} : \shp_\Ic(q_0) \to
\btPc_{\Ic_{q_0}}$ is surjective, monotone and continuous.
As in the previous section we can argue
that $\Ic_{q_0}$ is subordinate to $\Ic$, and using 
 Lemma~\ref{lm:top} for this map
 we obtain that the set
$\tGamma_{q_0}^{-1} (\tPc_{\Ic_{q_0}}^0(\le h)) =\shp_\Ic(q_0)\cap
\oPc_\Ic(=h)$ is connected.
Notice that the topological entropy of maps in $\shp_\Ic(q_0)$ is at
least $h=\htop(q_0)$, so the sets $\shp_\Ic(q_0)\cap
\oPc_\Ic(=h)$ and $\shp_\Ic(q_0)\cap
\oPc_\Ic(\le h)$ coincide.

Now we are going to construct a sequence of maps $q_i\in \Lc_\Ic$ converging to
$q_0$ which satisfies the following properties:
\begin{enumerate}
\item All turning points of $q_i$ are distinct and
  non-periodic. Because of the polynomial rigidity this implies that
  for any $q_i$ there exists unique $f_i\in \oPc_\Ic$ which is
  semi-conjugate to $q_i$.
\item The dynamics of turning points of $q_i$ will prevent the
  critical points of $f_i$ to collapse in the limit. It will be
  clearer later what this means precisely.
\end{enumerate}

\begin{figure}[t]
  \centering
  \setlength{\unitlength}{3315sp}%
\begingroup\makeatletter\ifx\SetFigFont\undefined%
\gdef\SetFigFont#1#2#3#4#5{%
  \reset@font\fontsize{#1}{#2pt}%
  \fontfamily{#3}\fontseries{#4}\fontshape{#5}%
  \selectfont}%
\fi\endgroup%
\begin{picture}(8033,4365)(169,-3694)
{\color[rgb]{0,0,0}\thinlines
\put(2071,-1771){\circle*{50}}
}%
{\color[rgb]{0,0,0}\put(1261,-3661){\circle*{50}}
}%
{\color[rgb]{0,0,0}\put(3151,-3661){\circle*{50}}
}%
{\color[rgb]{0,0,0}\put(181,-3661){\framebox(4320,4320){}}
}%
{\color[rgb]{0,0,0}\put(181,-3661){\line( 1, 1){4320}}
}%
\thicklines
{\color[rgb]{0,0,0}\put(3151,389){\line( 1,-2){810}}
}%
{\color[rgb]{0,0,0}\put(541,-1951){\line(-1,-2){270}}
}%
{\color[rgb]{0,0,0}\put(3961,-1231){\line( 1, 2){360}}
}%
\thinlines
{\color[rgb]{0,0,0}\multiput(5400,-2852)(9.00000,0.00000){311}{\makebox(1.9844,13.8906){\tiny.}}
\multiput(5400,-62)(9.00000,0.00000){311}{\makebox(1.9844,13.8906){\tiny.}}
\multiput(5400,-2852)(0.00000,9.00000){311}{\makebox(1.9844,13.8906){\tiny.}}
\multiput(8190,-2852)(0.00000,9.00000){311}{\makebox(1.9844,13.8906){\tiny.}}
}%
{\color[rgb]{0,0,0}\multiput(5400,-2852)(6.36986,6.36986){439}{\makebox(1.9844,13.8906){\tiny.}}
}%
{\color[rgb]{0,0,0}\put(5400,-2852){\line( 1, 2){540}}
}%
{\color[rgb]{0,0,0}\put(8190,-62){\line(-1,-2){360}}
\put(7830,-782){\line(-1, 2){270}}
\put(7560,-242){\line(-1,-2){720}}
\multiput(6840,-1682)(-4.50000,9.00000){41}{\makebox(1.9844,13.8906){\tiny.}}
\put(6660,-1322){\line(-1,-2){450}}
\put(6210,-2222){\line(-3, 5){270}}
}%
{\color[rgb]{0,0,0}\put(3458,-249){\line(-1,-2){408.800}}
\multiput(3052,-1068)(-4.44681,8.89362){48}{\makebox(1.9844,13.8906){\tiny.}}
\put(2843,-650){\line(-1,-2){557.400}}
\multiput(2286,-1765)(-4.50323,9.00645){32}{\makebox(1.9844,13.8906){\tiny.}}
\put(2146,-1486){\line(-1,-2){348.400}}
\multiput(1798,-2183)(-5.22132,8.70220){41}{\makebox(1.9844,13.8906){\tiny.}}
}%
{\color[rgb]{0,0,0}\put(1041,-2938){\line( 1, 2){550.800}}
}%
{\color[rgb]{0,0,0}\put(3218,253){\line(-1,-2){947}}
\multiput(2271,-1641)(-4.74545,9.49091){12}{\makebox(1.9844,13.8906){\tiny.}}
\multiput(2220,-1536)(-4.50968,-9.01935){32}{\makebox(1.9844,13.8906){\tiny.}}
\multiput(2079,-1815)(-4.88571,9.77143){8}{\makebox(1.9844,13.8906){\tiny.}}
\multiput(2046,-1746)(-4.58947,-9.17895){20}{\makebox(1.9844,13.8906){\tiny.}}
\multiput(1958,-1920)(-5.28529,8.80882){11}{\makebox(1.9844,13.8906){\tiny.}}
}%
{\color[rgb]{0,0,0}\put(1189,-3257){\line( 1, 2){713.200}}
}%
{\color[rgb]{0,0,0}\multiput(1801,-2041)(9.00000,0.00000){61}{\makebox(1.9844,13.8906){\tiny.}}
\multiput(1801,-1501)(9.00000,0.00000){61}{\makebox(1.9844,13.8906){\tiny.}}
\multiput(1801,-2041)(0.00000,9.00000){61}{\makebox(1.9844,13.8906){\tiny.}}
\multiput(2341,-2041)(0.00000,9.00000){61}{\makebox(1.9844,13.8906){\tiny.}}
}%
\thicklines
{\color[rgb]{0,0,0}\put(1261,-3391){\line(-1, 2){720}}
}%
\thinlines
{\color[rgb]{0,0,0}\multiput(1171,-2671)(9.00000,0.00000){241}{\makebox(1.9844,13.8906){\tiny.}}
\multiput(1171,-511)(9.00000,0.00000){241}{\makebox(1.9844,13.8906){\tiny.}}
\multiput(1171,-2671)(0.00000,9.00000){241}{\makebox(1.9844,13.8906){\tiny.}}
\multiput(3331,-2671)(0.00000,9.00000){241}{\makebox(1.9844,13.8906){\tiny.}}
}%
{\color[rgb]{0,0,0}\put(3511,-2491){\vector(-2, 1){1440}}
}%
\thicklines
{\color[rgb]{0,0,0}\put(1261,-3391){\line( 1, 2){1890}}
}%
\thinlines
{\color[rgb]{0,0,0}\put(3511,-2041){\vector(-2, 1){1260}}
}%
{\color[rgb]{0,0,0}\put(3511,-1411){\vector(-2, 1){900}}
}%
{\color[rgb]{0,0,0}\put(3511,-1051){\vector(-2, 1){360}}
}%
\put(3511,-2491){\makebox(0,0)[lb]{\smash{{\SetFigFont{10}{12.0}{\rmdefault}{\mddefault}{\updefault}{\color[rgb]{0,0,0}$(t,q_0^n(t))$}%
}}}}
\put(3511,-2041){\makebox(0,0)[lb]{\smash{{\SetFigFont{10}{12.0}{\rmdefault}{\mddefault}{\updefault}{\color[rgb]{0,0,0}$q_0^n$}%
}}}}
\put(3511,-1411){\makebox(0,0)[lb]{\smash{{\SetFigFont{10}{12.0}{\rmdefault}{\mddefault}{\updefault}{\color[rgb]{0,0,0}$q_{i_2}^n$}%
}}}}
\put(3511,-1051){\makebox(0,0)[lb]{\smash{{\SetFigFont{10}{12.0}{\rmdefault}{\mddefault}{\updefault}{\color[rgb]{0,0,0}$q_{i_1}^n$}%
}}}}
\put(7066,-1591){\makebox(0,0)[lb]{\smash{{\SetFigFont{10}{12.0}{\rmdefault}{\mddefault}{\updefault}{\color[rgb]{0,0,0}$Q$}%
}}}}
\put(1261,-3616){\makebox(0,0)[lb]{\smash{{\SetFigFont{10}{12.0}{\rmdefault}{\mddefault}{\updefault}{\color[rgb]{0,0,0}$c_l$}%
}}}}
\put(3151,-3616){\makebox(0,0)[lb]{\smash{{\SetFigFont{10}{12.0}{\rmdefault}{\mddefault}{\updefault}{\color[rgb]{0,0,0}$c_r$}%
}}}}
\end{picture}%
  \caption{Perturbations of $q_0$.}
  \label{fig:1}
\end{figure}
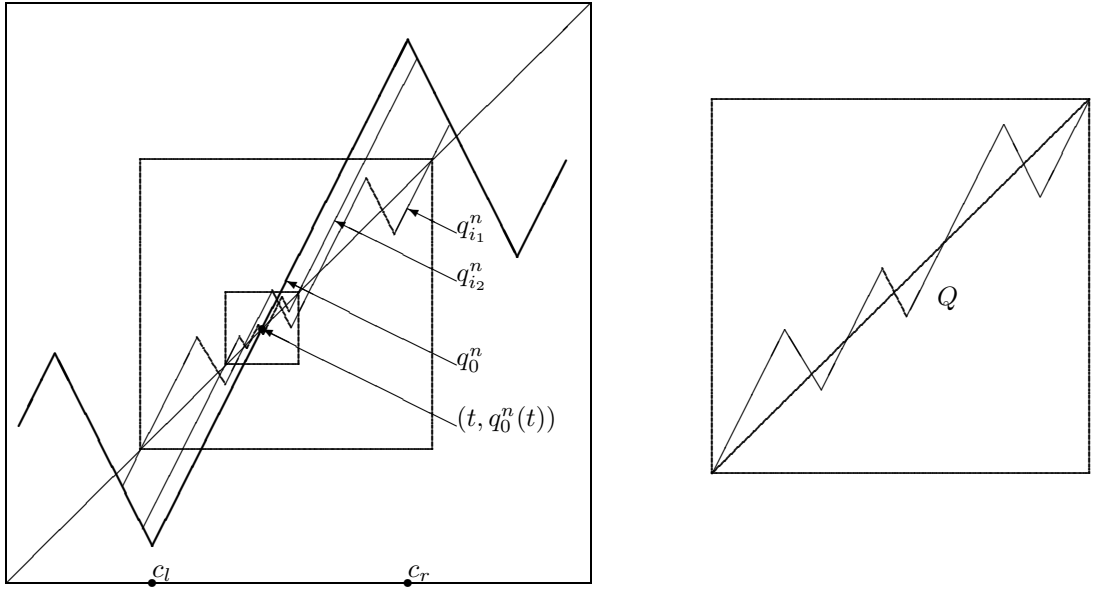

First, make a piece-wise linear map $Q : [0,1] \to \R$ with constant
slopes equal to $\pm \exp(h)$ which fixes the boundary points $0$,
$1$, and has exactly $k_t$ turning points
$t_1,\ldots,t_{k_t}$. Moreover, one can construct $Q$ in such a way
that none of the turning points is fixed by $Q$ and the positions of
the turning values on the graph of $Q$ are oscillating around the
diagonal, i.e. if $Q(t_i)> t_i$, then $Q(t_{i+1}) < t_{i+1}$, and if
$Q(t_i)< t_i$, then $Q(t_{i+1}) > t_{i+1}$. The last property is
equivalent to the following: each interval $[t_i,t_{i+1}]$ contains
a fixed point of $Q$.  An example of such a map is shown on
Figure~\ref{fig:1} on the right. Notice that the turning values are
not required to belong to the interval $[0,1]$.

In a short while we are going to use the following property of the map
$Q$: for any two consecutive turning points $t_i$ and $t_{i+1}$ the
interval $[Q(t_i),Q(t_{i+1})]$ contains one of these turning
points. Indeed, the interval $[t_i, t_{i+1}]$ contains a fixed point
of $Q$, so does $[Q(t_i),Q(t_{i+1})]$. If the last interval does not
contain $t_i$ and $t_{i+1}$, then $|Q(t_i)-Q(t_{i+1})|<|t_i-t_{i+1}|$
and the slope of the corresponding branch is less than one which is a
contradiction.

Now we are ready to construct the sequence $q_i$. Let $c_l$ and $c_r$
be the turning points of $q_0$ to the left and right of $t$ (or one of
$c_l$, $c_r$ can be a boundary point of $\Dom(q_o)$ if there is no
turning point there). We are going to glue a scaled copy of $Q$ at the
point $t$ and shift branches of $q_0$ defined on the intervals
$[c_r, t]$ and $[t,c_r]$ up and down, see Figure~\ref{fig:1}. More
precisely, we take the graph of the map
$x \mapsto \exp(-h(n-1))\, K^{-1} Q(Kx)$ defined on $[0,K^{-1}]$ for
large values of $K$, place it on the graph of $q_0$ around the point
$(t,q_0(t))$, and then adjust branches of $q_0$ in such a way that we
obtain a graph of a map in $\Lc_\Ic(=h)$. By taking a small
perturbation of the obtained map if necessary we can assume that all
its turning points are non-periodic. (Indeed, all maps in $\Lc_\Ic(=h)$
having a periodic turning point lie on a countable number of
codimension one planes, so maps without periodic turning points are
dense in $\Lc_\Ic(=h)$.) A sequence of maps in $\Lc_\Ic(=h)$ obtained in
this way for larger and larger values of $K$ and tending to $q_0$ we
will denote by $q_i$.

The polynomials in $\Pc_\Ic$ semi-conjugate to $q_i$ will be denoted
by $f_i$. Once again, because of the rigidity these polynomials are
unique. Since all turning points of maps $q_i$ are distinct, the
critical points of $f_i$ are distinct as well, so they are
quadratic. Moreover, in this case the semi-conjugacies between $f_i$
and $q_i$ are, in fact, just conjugacies. This implies that the
topological entropies of $f_i$ and $q_i$ coincide and all these
polynomials belong to $\pPc_\Ic(=h)$. By taking a subsequence we can
assume that the sequence $f_i$ converges to a polynomial
$f_0 \in \Pc_\Ic$.

By the continuity of the topological entropy we know that $ \htop(f_0)=h$.
From Lemma~\ref{lm:cont} it follows that $f_0$ is semi-conjugate to $q_0$
and, therefore, $f_0 \in \shp_\Ic(q_0)$. Now we will show that all
critical points of $f_0$ are distinct. It is obvious that if $c$ is a
critical point of $f_0$ which is mapped onto a simple turning point of
$q_0$, it has to be quadratic. Next, suppose that $f_0$ has a degenerate
critical point $c_*$ which is mapped onto $t$ by the
semi-conjugacy. Consider two cases.

Case 1: the point $c_*$ is periodic of period $n$, that is
$f_0^n(c_*)=c_*$ (recall that $t$ is also periodic of period
$n$). Then $c_*$ is a superattractor of $f_0$ and all polynomials
sufficiently close to $f_0$ will have their critical points close to
$c_*$ converge to a periodic attractor. This means that for each
sufficiently large value of $i$ the map $f_i$ has a critical point
with periodic itinerary, hence, the map $q_i$ has a periodic turning
point, which is a contradiction.

Case 2: $f_0^n(c_*) \neq c_*$. Then there exists a small interval $[a,b]$
containing the point $c_*$ such that $f_0^n([a,b])$ does not intersect
$[a,b]$. Take sufficiently
large $i$ so that the interval $[a,b]$ contains at least two critical
points $c_1$, $c_2$ of $f_i$ and so that $f_i^n([c_1,c_2]) \cap
[c_1,c_2] = \emptyset$. This is a contradiction because we checked
that $Q([t_k,t_{k+1}])$ contains either $t_k$ or $t_{k+1}$, a similar
property holds for $q_i$ because of its construction and maps $f_i$
and $q_i$ are topologically conjugate.

So, we have proved that $f_0 \in \oPc_\Ic(=h)$. Combining this and the
facts that $f_0 \in \shp_\Ic(q_0)$, that the set $\pPc_\Ic(=h)$ is
connected and the sequence $f_i\in
\pPc_\Ic(=h)$ converges to $f_0$ and that the set $\shp_\Ic(q_0)\cap
\oPc_\Ic(=h)$ is connected we get that the sets
 $\pPc_\Ic(=h)$ and $ \shp_\Ic(q_0)\cap \oPc_\Ic(=h)$
 cannot be separated, so
 $$\pPc_\Ic(=h) \cup \left(\shp_\Ic(q_0)\cap \oPc_\Ic(=h)\right)$$
is connected. Equality~(\ref{eq:op}) implies that $\oPc_\Ic$ is
connected too.

\section{On Thurston's question}
\label{sec:tq}

In the rest of the paper we will argue that most likely the answer to
Thurston's question is negative. 

In the arguments which follow we will not need multi-interval maps, so
we set $N=1$ from now on. Consider the space $\Lc_\Ic$ for some combinatorial information
$\Ic=\{1,\sigma,l,\sign\}$. It is clear that this space is
parameterised by $|l|$ parameters. Using notation of
Section~\ref{sec:multi-interval-multi} these parameters are the
entropy $h$ and the coefficients $b_1^i$ where
$i=1,\ldots,l(1)-1$. Notice that $b_1^0$ and $b_1^{l(1)}$ are fixed by
the boundary conditions. Since we are going to work with the case $N=1$ for
now we will drop the subscript $\cdot_1$ for the coefficients
$b_1^i$ and write $b^i$ instead. The
same applies to $l(1)$ and $s(1)$.

The turning points $0\le c^1\le \cdots \le c^l\le 1$ partition the
interval $I=[0,1]$ into $l+1$ open subintervals which we denote by
$J^0,\ldots,J^l$. Some of these intervals can be degenerate if some
turning points collide. Given a map $q\in \Lc_\Ic$ and a point $x \in
I$ we call an infinite sequence of symbols in
$\{c^1,\ldots,c^l,J^0,\ldots,J^l\}$ the \textit{itinerary} of $x$ if
the iterate $q^n(x)$ belongs to the corresponding element of the
sequence. Notice that if $q$ has collided turning points, the
itinerary of a point may be not unique, but this will not cause any
problems for us. The \emph{$n$-itinerary} of $x$ we will call the
sequence of the first $n+1$ elements of the itinerary which control
points $q^m(x)$ for $m=0,\ldots,n$.

The itinerary $\tilde \It=\{\tilde I_m\}$, $m=0,\ldots$, is called
\emph{compatible} with the itinerary $\It=\{I_m\}$ if the following
holds. For all $m$
\begin{enumerate}
\item if $\tilde I_m$ is one of the intervals $J^i$, then $I_m =
  \tilde I_m$;
\item if $\tilde I_m$ is one of the turning points, let it be $c^i$, then
  $I_m$ is either $c^i$ or $J^{i-1}$ or $J^i$.
\end{enumerate}

Take a map $q\in \Lc_\Ic$ which has a turning point $c^{i_0}$ which is
mapped to another turning point $c^{i_1}$ by some iterate $q^n$ and
let the orbit $\{q^m(c^{i_0}), \, m=1,\ldots, n-1\}$ not contain other
turning points. We do allow the case $i_0=i_1$ where the turning point
becomes periodic.  Denote the $n$-itinerary of $c^{i_0}$ by $\It^{i_0}$. It
it easy to see by a direct computation that the equation
$q^n(c^{i_0})=c^{i_1}$ has the form
\begin{equation} 
\sum_{i=1}^{l-1} Q^{\It^{i_0}}_i(e^h)b^i = Q^{\It^{i_0}}_0(e^h),\label{eq:6}
\end{equation}
where $Q^{\It^{i_0}}_i$ are some polynomials with rational
coefficients. These polynomials have some particular structure which
we will discuss in Section~\ref{sec:ordinary-isentropes}.  Here we
also used the equalities $c^i=\frac 12 (-1)^i s e^{-h} (b^{i-1}-b^i)$
and $q(c^{i})=\frac 12 (b^{i-1}+b^i)$.

This equation we will call the \emph{bifurcation} equation of
$\It^{i_0}$ and the polynomials $Q^{\It^{i_0}}_i$ will be called
\emph{bifurcation} polynomials. Notice that the bifurcation
equation is always well defined for periodic turning points.

Obviously, if another map $\tilde q\in \Lc_\Ic$ has a turning point with the
same $n$-itinerary as the turning point of $q$ under consideration,
then the parameters of this map satisfy equation~(\ref{eq:6}). Notice
that even if the $n$-itinerary of this turning point is just
compatible with $\It^{i_0}$, then the parameters of $\tilde q$ have to satisfy
equation~(\ref{eq:6}). This is an important observation which deserves
to be formulated as a lemma:

\begin{lemma}
  Let $q$ and $\tilde q$ be in $\Lc_\Ic$, $c^{i_0}$ and
  $\tilde c^{i_0}$ be their turning points with $n$-itineraries
  $\It^{i_0}$, $\tilde \It^{i_0}$. Moreover, let $\tilde \It^{i_0}$ be
  compatible with $\It^{i_0}$ and $q^n(c^{i_0})$ be a turning point of
  q (so the bifurcation equation is defined)\footnote{Also notice that
    the compatibility condition implies that in this case $\tilde
    q(\tilde c^{i_0})$ is also a turning point}. Then the parameters of
  the map $\tilde q$ satisfy the bifurcation equation of the map $q$:
  $$
  \sum_{i=1}^{l-1} Q^{\It^{i_0}}_i(e^{\tilde h})\tilde b^i =
  Q^{\It^{i_0}}_0(e^{\tilde h}).
  $$
\end{lemma}
Notice that in the lemma above the maps $q$ and $\tilde q$ can have
different topological entropies.

Let us make clear that the converse of this lemma does not hold. If
for some map its parameters satisfy equation~(\ref{eq:6}), it does not imply
that the corresponding turning point has the given $n$-itinerary: one
would have to consider a bunch of inequalities together with equation
(\ref{eq:6}) to guaranty that all the points from the orbit of the
turning point fall into appropriate intervals as the itinerary
dictates. However, maps close to $q$ and satisfying the bifurcation
equation do have a turning point with the same $n$-itinerary as the
turning point $c^{i_0}$ of $q$ as the following lemma claims.

\begin{lemma} \label{lm:neib}
  Let $c^{i_0}$ be a turning point of $q\in \Lc_\Ic$ such that
  $q^n(c^{i_0})$ is also a turning point and $n\ge 1$ is minimal with
  this property. Then there exists a neighbourhood of $q$ in the space
  $\Lc_\Ic$ such that every map in this neighbourhood satisfying the
  corresponding bifurcation equation has a turning point
  with $n$ itinerary coinciding with $n$-itinerary of $c^{i_0}$ of the
  map $q$. 
\end{lemma}

Remark. Notice that if $q^n(c^{i_0}) = c^{i_0}$, i.e. $c^{i_0}$ is a
periodic turning point, then all the maps in the neighbourhood given
by the lemma and satisfying the bifurcation equation will have a
periodic turning point of period $n$ with the same itinerary.

\begin{proof}
  Let $\{c^{i_0}, J^{m_1},\ldots,J^{m_{n-1}}\}$ be the $n-1$ itinerary
  of $c^{i_0}$. We know that $q^j(c^{i_0})$ is not a turning point for
  $j=1,\ldots,n-1$, so it belongs to the interior of
  $J^{m_j}$. Therefore, there exists a neighbourhood of $q$ in the
  space $\Lc_\Ic$ such that if a map belongs to this neighbourhood,
  then the $n-1$ itinerary of the corresponding to $c^{i_0}$ turning
  point is $\{c^{i_0}, J^{m_1},\ldots,J^{m_{n-1}}\}$. Then the
  bifurcation equation ensures that the $n$ itinerary of this point
  will be $\{c^{i_0}, J^{m_1},\ldots,J^{m_{n-1}}, c^{i_1}\}$ where
  $c^{i_1}=q^n(c^{i_0})$.
\end{proof}

The equation~(\ref{eq:6}) is linear in all $b^i$ and as such it is easy
to solve. There are several cases to consider:

\textbf{Case 1.} For given $h$ some of the polynomials $Q^{\It^{i_0}}_i$,
$i=1,\ldots,l-1$, are non-zero at the point $e^h$. Then the parameters
of maps in $\Lc_\Ic(=h)$ satisfying equation~(\ref{eq:6}) form $l-2$
dimensional linear space. This case might be regarded as ``generic''.

\textbf{Case 2.} For given $h$ we have $Q^{\It^{i_0}}_i(e^h)=0$ for all
$i=0,1,\ldots,l-1$. Clearly, all parameters of maps in $\Lc_\Ic(=h)$
satisfy equation~(\ref{eq:6}). This is a very special case.

\textbf{Case 3.} For given $h$ all the polynomials $Q^{\It^{i_0}}_i$,
$i=1,\ldots,l-1$ vanish at $e^h$, but $Q^{\It^{i_0}}_0(e^h)\neq 0$. There are no
maps in $\Lc_\Ic(=h)$ which have the turning point $c^{i_0}$ with the
given itinerary.

These different cases motivate the following definition:
\begin{definition}
  Let $c^{i_0}$ be a turning point of $q\in \Lc_\Ic$ such that
  $q^n(c^{i_0})$ is also a turning point and $n\ge 1$ is minimal with
  this property. Then this turning point is called \emph{ordinary} if
  some of the polynomials $Q^{\It^{i_0}}_i$, $i=1,\ldots,l-1$ do not
  vanish at $e^{\htop(q)}$ (so we are in Case 1). If
  $Q^{\It^{i_0}}_i(e^{\htop(q)})=0$ for all $i=0,\ldots,l-1$, then the
  turning point $c^{i_0}$ is called \emph{exceptional} (Case 2 above).
\end{definition}

Remark. Since all the polynomials $Q^{\It^{i_0}}_i$ have rational
coefficients, it is clear that if an isentrope of entropy level $h$
has an exceptional turning point, then the number $e^h$ is
algebraic. In particular, only countably many isentropes can have
exceptional turning points.

Soon we will give some examples of ordinary and exceptional turning
points, however before that let us demonstrate their relevance to
Thurston's question. We need another definition first.

\begin{definition}
  A turning point of a map $q \in \Lc_\Ic$ is called \emph{controlled} if it is
  periodic or is mapped onto a periodic turning point by some iterate
  of $q$.

  A map $q\in \Lc_\Ic$ is called a \emph{codimension one hyperbolic}
  map if it has one turning point whose orbit does not contain any
  turning points, and all other $l-1$ turning points $c^{i_j}$,
  $j=1,\ldots,l-1$, are controlled. Moreover, if the determinant of
  the matrix $\|Q^{i_j}_i(e^{\htop(q)})\|$, $i,j = 1,\ldots,l-1$
  formed by the bifurcation polynomials is non-zero, such the map
  $q$ will be called an \emph{ordinary codimension one hyperbolic} map.

  Similarly, a critical point of a map $p \in \Pc_\Ic$ is called
  \emph{controlled} if it is contained in the basin of a periodic
  attracting point.

  A polynomial $p\in \Pc_\Ic$ is called a \emph{codimension one
    hyperbolic} map if its all periodic points are hyperbolic and it
  has exactly $l-1$ controlled critical points counted with the
  multiplicities.  Moreover, if $p$ is semi-conjugate to an ordinary
  codimension one hyperbolic map $q \in \Lc_\Ic(=\htop(p))$, then $p$
  will be called an \emph{ordinary codimension one hyperbolic} map.
\end{definition}

Notice that a codimension one hyperbolic map is not hyperbolic! It has
one critical point whose iterates do not converge to a periodic
attractor. 

\begin{lemma} \label{lm:r} Let $q\in \Lc_\Ic$ be an ordinary
  codimension one hyperbolic map and $c^{i_j}$, $j=1,\ldots,l-1$ be
  its controlled turning points. Then there exist an interval
  $(h_-,h_+)$ containing $\htop(q)$ and a function
  $r: (h_-,h_+) \to \Lc_\Ic$ such that
  \begin{itemize}
  \item $\htop(r(h))=h$ for all $h \in (h_-,h_+)$;
  \item the parameters $b^i$ of the map $r(h)$ are given by some rational
    functions $R_i(e^h)$; 
  \item $r(\htop(q))=q$;
  \item for all $h \in (h_-,h_+)$ the itineraries of the controlled
    turning points $c^{i_j}$ of $r(h)$ coincide with the itineraries
    of the corresponding turning points of the map $q$;
  \item the converse also holds: if the itineraries of $l-1$ turning points
    of a map $q'\in \Lc_\Ic$ are compatible with the itineraries of the
    corresponding controlled turning points of $q$ and $\htop(q')\in (h_-,h_+)$,
    then $q'=r(\htop(q'))$.
  \end{itemize}

  In particular, there are no other than $q$ maps in $\Lc_\Ic(=\htop(q))$ which
  have the controlled turning points with the itineraries compatible with the
  itineraries of the controlled turning points of $q$.
\end{lemma}

\begin{proof}
  From the previous discussion we already know that if the itineraries
  of the controlled turning points of some map $q'$ are the same (or
  compatible) as of $q$, then the parameters of $q'$ must satisfy the
  $l-1$ bifurcation equations. Notice that we have $l-1$ linear in
  $b$ equations which also depend on the parameter $h$. Since the map
  $q$ is ordinary, the solution of this system of bifurcation
  equations as a function of $h$ is well defined in some interval
  around the point $\htop(q)$. Using Lemma~\ref{lm:neib} and by
  shrinking this interval if necessary we can ensure that maps
  corresponding to the solutions of this system have the controlled
  points with the given itineraries.
\end{proof}

\begin{mytheorem} \label{th:D} Let $p\in \Pc_\Ic$ be an ordinary
  codimension one hyperbolic polynomial map of positive entropy. Then
  $p$ cannot be approximated by hyperbolic polynomial maps of the same
  entropy $\htop(p)$.
\end{mytheorem}

Remark 1. There is nothing special about the polynomial space here, this
theorem also holds for the space $\bMc_\Ic^1$.

Remark 2. If one drops the condition that the map is ordinary, the
theorem does not hold anymore. Once we construct maps with exceptional
critical points in Sections~\ref{sec:except-isentr-from}
and~\ref{sec:casc-except-itin}, one can easily find codimension one
hyperbolic maps which can be approximated by hyperbolic maps of the
same entropy.

Remark 3. We will see (Theorem~\ref{th:no_exceptional}) that if
the entropy is larger than $\log 3$, then the corresponding isentrope
can contain only ordinary codimension one hyperbolic maps. In the such
case if one finds a codimension one hyperbolic map, they do not need
to check that it is ordinary, it holds automatically.

\begin{proof}
  This Theorem is a consequence of Lemma~\ref{lm:r}.

  From the definition of ordinary maps we can find an ordinary
  codimension one hyperbolic map $q\in \Lc_\Ic(=\htop(p))$
  semi-conjugate to $p$. Since all attracting periodic points of $p$ are
  hyperbolic, there exists a neighbourhood of $p$ in $\Pc_\Ic$
  where these attracting points persist and the critical points of
  maps in this neighbourhood corresponding to the controlled critical
  points of $p$ are also controlled and are in the basins of
  attraction of the corresponding periodic points. Let $p'$ be in this
  neighbourhood and $q' \in \Lc_\Ic$ be semi-conjugate to $p'$ and of
  the same entropy $\htop(q')=\htop(p')$. Let $c^{i_0}(p)$ be one of
  the controlled critical points of $p$, and $c^{i_0}(p')$,
  $c^{i_0}(q)$ be the corresponding critical (turning) points of $p'$,
  $q'$. It is easy to see that the itinerary of $c^{i_0}(q')$ is
  compatible with the itinerary of $c^{i_0}(q)$. Now assume that
  $\htop(p')=\htop(p)$.  Since the map $q$ is ordinary and due to
  Lemma~\ref{lm:r} we know that if $\htop(p')=\htop(p)$ and,
  therefore, $\htop(q')=\htop(q)$, then $q'$ and $q$ are the same
  maps. One of the turning points of $q$ is not eventually periodic,
  hence the map $p'$ has a critical point which is not in the basin of
  attraction of some periodic attractor. So, the map $p'$ cannot be
  hyperbolic.
\end{proof}

\section{Exceptional isentropes}
\label{sec:except-isentr}

In this section we study with more details when an isentrope can have
an exceptional turning point and partially answer on
Question~\ref{q:2}. We start with a number of examples. We restrict
ourselves to the case of bimodal maps which can be easily
generalised. To make computations simpler we will rescale the domain
of the definition of maps we consider so our bimodal maps are defined
by this formula:
\begin{equation}
  \label{eq:8}
q=q_{\lambda,b} : x \mapsto \left\{
\begin{array}{ll}
  \phantom{-}\lambda x + 1, & \mbox{if } x \in J^0=[-a, c^1]\\
  -\lambda x +b, & \mbox{if } x \in J^1=[c^1, c^2] \\
  \phantom{-}\lambda x - 1, & \mbox{if } x \in J^2=[c^2, a]\\
\end{array}
\right.
\end{equation}
where $\lambda = e^h$, $a=\frac{1}{\lambda - 1}$,
$c^1=\frac{b-1}{2\lambda}$, $c^2=\frac{b+1}{2\lambda}$. Also, notice
that $\pm a$ are fixed points of $q$ and that $q(c^1)=\frac 12 (b+1)$, $q(c^2)=\frac 12 (b-1)$. We want $q$ to
map the interval $[-a,a]$ into itself, this implies that $\lambda$ and
$b$ should satisfy inequalities $\lambda \in [1,3]$ and $ b \in
[-\frac{3-\lambda}{\lambda-1},\frac{3-\lambda}{\lambda-1}]$.

\subsection{Exceptional isentropes from unimodal tent maps.}
\label{sec:except-isentr-from}

The simplest examples of exceptional isentropes can be constructed
using unimodal tent maps with a periodic turning point.

Fix some parameter $\lambda$ in the interval $(1,2)$. For such
$\lambda$ there exists a non-degenerate interval of parameters $b$
(which is $[-\frac{3-\lambda}{\lambda-1}, -1]$) such that
$q_{\lambda,b}$ maps the interval $[-a, c^2]$ inside itself. The
itineraries of the turning point $c^1$ for all values of $b$ in this
interval are the same and coincide with the itinerary of the turning
point of the unimodal tent map of entropy $h=\log \lambda$.

Now fix $\lambda=e^h \in (1,2)$ in such a way that the unimodal tent
map of entropy $h$ has a periodic turning point. Then in the bimodal
family maps $q_{\lambda,b}$ will have a periodic turning point $c_1$
of the same itinerary for all
$b\in [-\frac{3-\lambda}{\lambda-1}, -1]$. Because of Lemma~\ref{lm:r}
we know that for ordinary turning points we can have at most one
parameter $b$ for the given itinerary, therefore $c^1$ is exceptional.
On the other hand, when $b$ varies in this interval, the itinerary of
the other turning point $c^2$ is not constant and there are
infinitely many different itineraries of $c^2$ when this turning point
becomes preperiodic. Using arguments similar to ones we use in the
first part of the paper (in particular, continuity of the map
$\tLambda$) one can show that the isentrope $\Pc_\Ic(=\log \lambda)$ contains
infinitely many combinatorially different hyperbolic maps and that
this isentrope contains a codimension one hyperbolic maps which can be
approximated by hyperbolic maps.

\subsection{Cascades of exceptional itineraries.}
\label{sec:casc-except-itin}

There is another mechanism which produces isentropes with exceptional
turning points and generalises the previous construction. We start
with a concrete example where most of the things can be explicitly
computed.

In the bimodal family under consideration let us consider maps which
have a periodic turning point $c^1$ of period 2 and with itinerary
$\It=\{c^1, J^2, c^1, \ldots\}$. One can easily compute the bifurcation
equation for this itinerary:
$$
(\lambda^2 -1) b = -(\lambda-1)^2,
$$
so $Q^{\It}_1(\lambda) = \lambda^2-1$, and
$Q^{\It}_0(\lambda) = -(\lambda-1)^2$. The case of $\lambda=1$ is
always special: it is easy to see that for any itinerary $\It$ one has
$Q^\It_1(1)=0$.  In our case we have $Q^\It_0(1)=0$ as well, so we can
reduce $\lambda -1$ factor and obtain
$$
(\lambda + 1) b = 1-\lambda.
$$
From this equation we can see that there is no exceptional isentropes
for the given itinerary $\It$ because the polynomials $\lambda+1$ and
$\lambda-1$ never vanish at the same time.

Nothing exciting so far. Now let us consider some other itinerary $\It'$
so that $\It$ is compatible to $\It'$. For example, let $\It'=\{c^1,J^2,
J^0,J^2,c^1,\ldots\}$. It is clear that if maps with such the
itinerary exist, then the turning point $c^1$ is periodic of period
$4$. Since $\It$ is compatible with $\It'$, all the solutions of the
bifurcation equation for $\It$ are also solutions of the bifurcation
equation for $\It'$. This implies that the bifurcation polynomials for
$\It'$ can be factorised as
\begin{eqnarray*}
  Q^{\It'}_1(\lambda) & = & F(\lambda) Q^{\It}_1(\lambda)\\
  Q^{\It'}_0(\lambda) & = & F(\lambda) Q^{\It}_0(\lambda)
\end{eqnarray*}
where $F$ is a polynomial. It is easy to see that the degree of the
polynomials $Q^{\It'}_{0,1}$ is $4$, so the degree of $F$ is two. A
simple computation (assisted by Wolfram Mathematica) gives
$$
F(\lambda) = \lambda^2+1.
$$
The roots of $F$ are complex, hence there are no other solutions of
the bifurcation equation for $\It'$ in the region of the interest
except $b=\frac{1-\lambda}{1+\lambda}$ and there are no bimodal maps
in $\Lc_\Ic$ which realise $\It'$.

If instead of the itinerary $\{c^1,J^2, J^0,J^2,c^1,\ldots\}$ we
considered $\It'=\{c^1,J^2, J^1,J^2,c^1,\ldots\}$, we could argue
again that the bifurcation polynomials must have a common factor which
can be computed to be $F(\lambda) = \lambda^2-1$. The roots of this
polynomial are real, but of no interest for us, so again, there are no
bimodal maps which realise $\It'$.

Let us move forward and find a nontrivial itinerary compatible to
$\It$ which can be realised by some bimodal maps. Consider the
itinerary $I'=\{c^1,J^2, J^1,J^2,J^0,J^2,c^1,\ldots\}$. Using the same
argument as before we can factorise its bifurcation polynomials and
compute the factor to be $F(\lambda) = \lambda^4-\lambda^2-1$. This
factor polynomial has two complex roots, one negative root and one
positive root
$\lambda_e=\sqrt{\frac{1}{2} \left(\sqrt{5}+1\right)} \approx
1.27202$. Furthermore, we can check that for this value of
$\lambda=\lambda_e$ if $-0.119726\leq b\leq 0.346014$, then the
turning point $c^1$ indeed has the itinerary $\It'$. Thus we have
found an exceptional itinerary.

Let us make an interesting observation. Because of the symmetry we
know that if the turning point $c^2$ is periodic with the itinerary
$\It''=\{c^2, J^0, J^1,J^2,J^1,J^0,c^2,\ldots\}$, then $\It''$ is
exceptional as well for parameters $\lambda=\lambda_e$ and $-0.346014
\leq b \leq 0.119726$. This implies that for all $b \in (-0.119726,
0.119726)$ both turning points are periodic of period $6$ with the
constant itineraries, and therefore all maps in this parameter
interval are combinatorially and topologically conjugate, see Figure~\ref{fig:2}. We
investigate this phenomenon in more details in the next section.

\begin{figure}[t]
  \centering
  \begin{tabular}{ll}
    \includegraphics[scale=0.8]{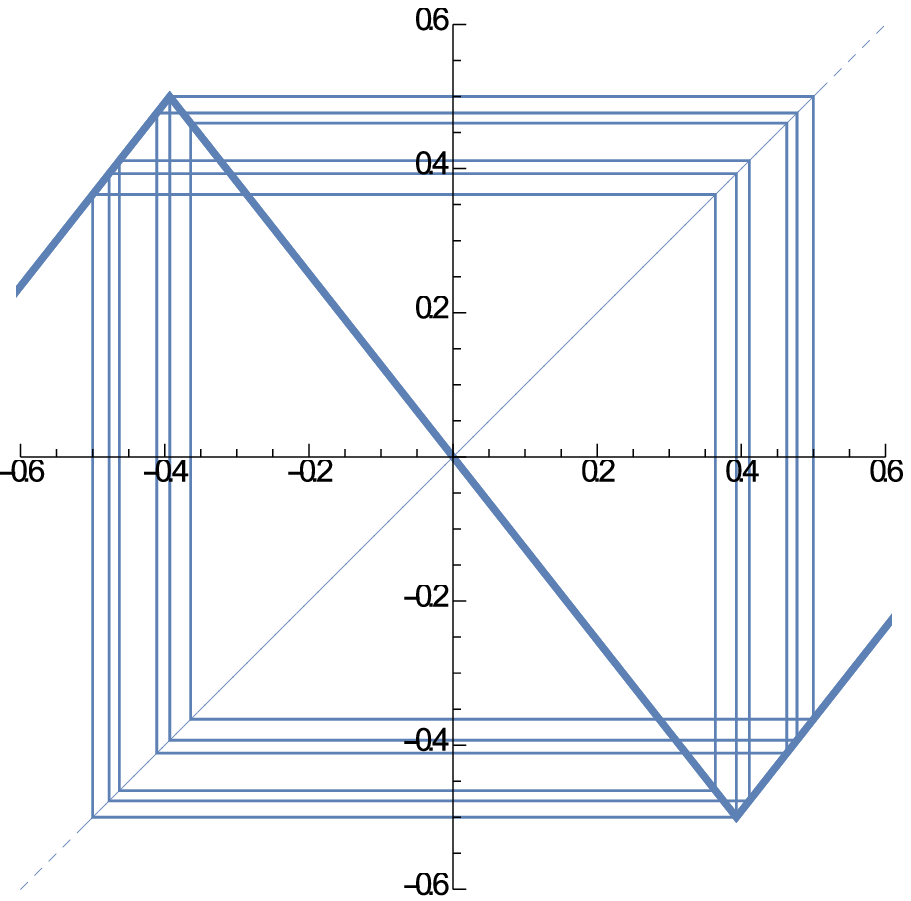}
    &
    \includegraphics[scale=0.8]{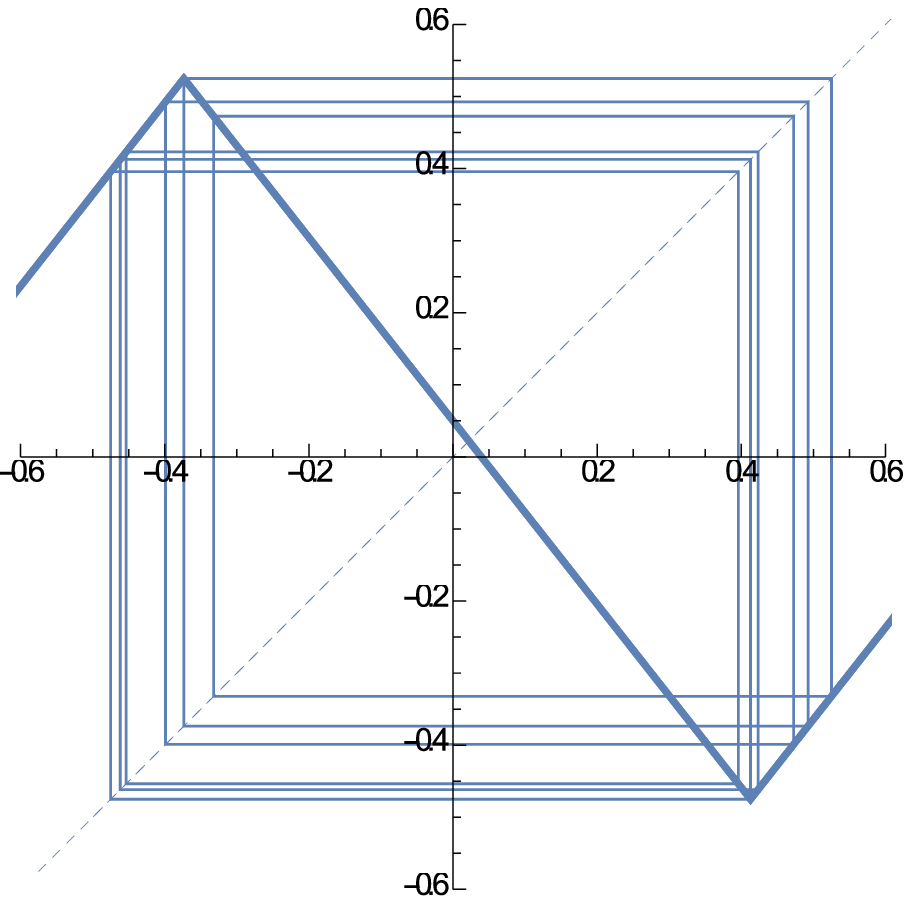}
  \end{tabular}
  \caption{Iterates of the turning points when $\lambda=\lambda_e$,
  $b=0$ (left) and $b=0.05$ (right) }
  \label{fig:2}
\end{figure}

The method of finding exceptional itineraries can obviously applied
not only to the itinerary $\It=\{c^1, J^2, c^1, \ldots\}$, but to
other periodic itineraries too. We have to do the following. Take some
periodic itinerary $\It=\{c^1, J^{i_1},\ldots, J^{i_m}, c^1,
\ldots\}$ and check that this itinerary can be realised by some maps
in $\Lc_\Ic$. Form another compatible itinerary $\It'=\{c^1,
J^{i_1},\ldots, J^{i_m}, J^{j_1},J^{i_1},\ldots,
J^{i_m},J^{j_2},\ldots,J^{i_1},\ldots, J^{i_m},c^1,\ldots\}$, where
all $j_1, j_2, \ldots$ are either $0$ or $1$. As
before, the bifurcation polynomials of $\It'$ have a common factor. If
this factor has real roots in the interval $(1,3)$, investigate if for
values of these roots the corresponding maps can realise $\It'$. If
they can, we have another exceptional itinerary.

In this way we can obtain many exceptional itineraries starting with
$\It=\{c^1,J^1,c^1,\ldots\}$. Another example would be
$\It=\{c^1,J^1,J^0,J^2,c^1,\ldots\}$. This itinerary is not
exceptional and its bifurcation equation is
$$
(\lambda - 1) \left(\lambda ^3+\lambda ^2+\lambda -1\right)b =(\lambda
- 1)\left(-\lambda ^3+\lambda ^2-\lambda -1\right).
$$
The compatible itinerary
$\It'=\{c^1,J^1,J^0,J^2,J^1,J^1,J^0,J^2,J^0,J^1,J^0,J^2,c^1,\ldots\}$
is exceptional, its bifurcation equation factor is
$$
F(\lambda)=\lambda ^8-\lambda ^4-1
$$
which has a root $\approx 1.12784$. For this value of $\lambda$ and
for $-0.808065\leq b\leq -0.720696$ the itinerary of $c^1$ under the
map $q_{\lambda,b}$ is $\It'$.

\subsection{Non-rigidity in the bimodal family $\Lc_\Ic$.}
\label{sec:non-rigidity-bimodal}

We know the following fundamental rigidity result for polynomials with
all critical points real: if two such polynomials are combinatorially
equivalent and do not have periodic attractors, then they are linearly
conjugate. In other words, if we consider a normalised
parameterisation of the polynomial family, there exists only one
parameter with this prescribed combinatorics. For the piece-wise
linear maps of constant slope a similar rigidity result holds provided
the maps are transitive, see \cite{Alseda2015}. If the transitivity
condition does not hold, the rigidity does not necessarily hold
either:
\begin{mytheorem}\label{th:nonrigit}
  There exists a nonempty open set $\Ec \subset \Lc_\Ic$ such that if
  the set $\Ec(=\htop)$ is not empty, then all the maps in
  $\Ec(=\htop)$ are combinatorially equivalent (and therefore
  topologically conjugate).
\end{mytheorem}

\begin{proof}
  Let us look at the example described in the previous section, and
  consider map $q=q_{\lambda_e, 0}$. We know that in this case both
  the turning points are periodic of period $6$, see
  Figure~\ref{fig:2}. Consider the interval $R^1$ defined as
  $[q^4(c^1),q^2(c^1)]$ and notice that $c^1 \in R^1$. It is easy to
  see that $q^2(R^1) \subset R^1$, so $R^1$ is a renormalization
  interval of period two. The interval $R^2=[q^2(c^2),q^4(c^2)]$ is
  another renormalization interval around the turning point $c^2$.

  For parameters $\lambda$ and $b$ close enough to $\lambda_e$ and $0$
  the intervals
  $R^i_{\lambda,b}=[q^4_{\lambda,b}(c^i),q^2_{\lambda,b}(c^i)]$,
  $i=1,2$, will still be renormalization intervals of period two.
  The maps $q^2_{\lambda,b}|_{R^i_{\lambda,b}}$ are unimodal tent maps
  and their combinatorics is completely determent by the parameter
  $\lambda$. Thus for fixed $\lambda$ close to $\lambda_e$ and all $b$
  close to zero all the maps $q_{\lambda,b}$ have the same combinatorics.
\end{proof}

\subsection{Non-existence of exceptional isentropes for large
  entropies.}
\label{sec:non-exist-except}

All examples of isentropes with exceptional turning points we had so
far have been given for the parameter $\lambda$ smaller than two. We
will prove that this is always the case in the bimodal case:

\begin{mytheorem}\label{th:no_exceptional}
  There do not exist isentropes with exceptional turning points in the
  space of bimodal maps $\Lc_{1,id, 2, s}$ of topological entropies
  larger than $\log 2$.

  In general, there do not exist exceptional isentropes of topological entropies
  larger than $\log 3$.
\end{mytheorem}

\begin{proof}
  We start the proof with the general case when $l$ is not necessarily
  two. We will be using the notation introduced in
  Section~\ref{sec:linear-model}, i.e. on the interval $J^i$ the map
  $q$ is defined as $q(x)= (-1)^i s \lambda x + b^i$. Recall that
  $b^0$ and $b^l$ are fixed by the boundary conditions.
  
  Let $\It=\{c^{i_0},J^{i_1}, \ldots, J^{i_{m-1}},c^{i_m},\ldots\}$ be
  an exceptional itinerary, i.e. there exist parameters $\lambda_0$
  and $b_0$ such that the itinerary of $c^{i_0}_{\lambda_0,b_0}$ under
  the map $q_{\lambda_0,b_0}$ is $\It$ and all the bifurcation
  polynomials of $\It$ vanish for $\lambda=\lambda_0$. It is also
  clear that if the parameter $b$ is close enough to $b_0$, then the
  itinerary of $c^{i_0}_{\lambda_0,b}$ will be $\It$ again. 

  Now consider the iterates of $c^{i_0}$ under the map $q$. From the
  definition it is easy to see that
  $$
  q_{\lambda_0,b}^k(c^{i_0}) = \sum_{i=0}^l w^i_k b^i
  $$
  where $w^i_k$ are some numbers (which in general depend on
  $\lambda_0$). These numbers are related by
  recursive formulas of the form
  $$
  w^i_{k+1} = (-1)^{i_k} s \lambda_0 w^i_k + \delta_{i_k}^i
  $$
  where $\delta_{i_k}^i$ is equal to one if $i=i_k$ and zero
  otherwise. The initial conditions for these recursive formulas are
  $$
  w^i_1 = \frac 12 \left(\delta^{i-1}_{i_0} +\delta^i_{i_0}\right)
  $$
  because $q(c^{i_0}) = \frac 12 (b^{i-1}+b^i)$.

  Fix some $\hat i \neq 0,l$ such that $w^{\hat i}_1=\frac 12$. Notice
  that if $|x|\ge \frac 12$, then $|\pm \lambda x| > \frac 32$ and
  $|\pm \lambda x + 1| > \frac 12$ for all $\lambda > 3$. This implies
  that
  \begin{equation}
    |w^{\hat i}_k| > \frac 12 \label{eq:7}
  \end{equation}
  for all $k \ge 2$.

  We know that $q^{m+1}(c^{i_0}) = q (c^{i_m})$. Since the level
  $\lambda_0$ is exceptional we also know that the bifurcation
  polynomial $Q^{\It}_{\hat i}(\lambda_0)=0$. This implies that
  $w^i_{m+1}$ is either $\frac 12$ or zero. This contradicts
  inequality~(\ref{eq:7}). Thus $\lambda_0$ cannot be larger than
  three.

  The case of bimodal maps where $l=2$ is dealt with similarly. We
  will consider the case $s=+1$, the other case  $s=-1$ is analogous.
  The index $\hat i$ here is just $1$, $w^1_1=\frac 12$ and the
  recursive formula is
  $$
  w^1_{k+1}=\left\{
    \begin{array}{ll}
    \phantom{-}\lambda_0 w^1_k &\mbox{ if  $i_k$ is $0$ or $2$}\\ 
    -\lambda_0 w^1_k + 1 &\mbox{ if  $i_k$ is $1$}
    \end{array}
  \right.
  $$
  This formula implies that if $\lambda_0>2$ and $w^1_k \in (-\infty,0)\cup[\frac 12,
  +\infty)$, then $w^1_{k+1} \in (-\infty,0)\cup(\frac 12,
  +\infty)$. So, arguing as in the general case we get a contradiction.
\end{proof}

\section{Codimension one hyperbolic maps in the bimodal family}
\label{sec:ordinary-isentropes}

Let us consider the family of real polynomial maps of degree 3 and
their isentropes of entropy larger than $\log 2$. From the previous
section we already know that there are no exceptional critical points
in this case, so Theorem~\ref{th:D} implies that if a map has
entropy larger than $\log 2$, two critical points, one of which is
not controlled and the other is controlled (so it is periodic), then
such a map cannot be approximated by hyperbolic maps of the same
entropy.

We conjecture that these codimension one hyperbolic maps exist on
every isentrope (with some trivial exceptions like $h=\log 3$ for the
bimodal maps). Let us see what would happen if this is not the case.

Fix some entropy level $h > \log 2$ and the corresponding
isentrope in the space of the bimodal piece-wise linear maps of the
constant slopes given by formula~(\ref{eq:8}). Suppose that $c^1$ is a
periodic point of period $n$. Then equation~(\ref{eq:6}) can be
written as
\begin{equation}
  \label{eq:9}
   Q^1_1(\lambda) b  =Q^1_0(\lambda),
\end{equation}
where $Q^1_0(\lambda) = \sum_{i=0}^n \alpha^1_i \lambda^i$ and
$Q^1_1(\lambda) = \sum_{i=0}^n \beta^1_i \lambda^i$ and the coefficients
$\alpha^1_i$ and $\beta^1_i$ can be explicitly computed if the itinerary
of $c^1$ is known. Moreover, these coefficients satisfy the following
conditions which are easy to obtain by a direct computation:
$-\alpha^1_n=\beta^1_n=1$, $\alpha^1_0=\beta^1_0 = \pm 1$, and for
$i=1,\ldots, n-1$ we have $\alpha^1_i$, $\beta^1_i \in \{-2,0,2\}$ and
$|\alpha^1_i|+|\beta^1_i|=2$.  The last condition means that if a
coefficient in front of $\lambda^i$ is non-zero in the polynomial
$Q^1_0$, then the corresponding coefficient in $Q^1_1$ must be zero and
vise verse (however they cannot be both zeros at the same time).

For example, let us consider the case when $\lambda$ is close to
$3$. Then the critical value $q(c^1)$ is close to the fixed repelling
point $a$ and there exists a bimodal map $q$ such that $c^1$ is
periodic of period $n$ with the itinerary
$\{c^1, \underbrace{J^2,J^2,\ldots,J^2}_{n-1},c^1,\ldots\}$ (recall
that $J^2=(c^2,a)$). Formula~(\ref{eq:9}) in this case becomes
$$
(\lambda^n-1) b = -\lambda^n + 2 \lambda^{n-1} +\cdots + 2\lambda -1.
$$

Similarly, if the critical point $c^2$ is periodic or is mapped onto
$c^1$ by some iterate of the map, the parameters $\lambda$ and $b$
satisfy 
\begin{equation}
  \label{eq:10}
   Q^2_1(\lambda) b  =Q^2_0(\lambda),
\end{equation}
where $Q^2_0(\lambda) = \sum_{i=0}^{n'} \alpha^2_i \lambda^i$ and
$Q^2_1(\lambda) = \sum_{i=0}^{n'} \beta^2_i \lambda^i$ and for the
coefficients the following holds: $\alpha^2_{n'}=\beta^2_{n'}=1$,
$\alpha^2_0=\pm 1$, $\beta^2_0 = \pm 1$, and for
$i=1,\ldots, {n'}-1$ we have $\alpha^2_i$, $\beta^2_i \in \{-2,0,2\}$ and
$|\alpha^2_i|+|\beta^2_i|=2$.

Suppose that all maps in some isentrope $\bPc_\Ic(=h)$, where
$h \in (\log 2, \log 3)$, can be approximated by hyperbolic maps. Fix
the corresponding value of $\lambda = e^h$ and consider the
corresponding isentrope $\bLc_\Ic(=h)$. Theorem~\ref{th:D} and
Theorem~\ref{th:no_exceptional} imply that $\bLc_\Ic(=h)$ does not
contain any codimension one hyperbolic maps. Then for any parameter
$b$ such that the turning point $c^1$ of
$q_{\lambda,b}\in \bLc_\Ic(h)$ is periodic (and, therefore,
equality~(\ref{eq:9}) holds) the other critical point $c^2$ must be
controlled as well and equality~(\ref{eq:10}) must holds.  This
implies that $\lambda$ satisfies the equality
\begin{equation}
  \label{eq:11}
  Q^1_0(\lambda)Q^2_1(\lambda) = Q^1_1(\lambda) Q^2_0(\lambda)
\end{equation}
The parameter $\lambda$ must satisfy an equality of this type whenever
one of the turning points of $q_{\lambda,b}$ is periodic. For this
fixed $\lambda$ there are
infinitely many different values of $b$ when this map has a periodic
critical point, therefore $\lambda$ has to satisfy infinitely many
different polynomial equalities of type~(\ref{eq:11}). Notice that the
involved polynomials $Q^i_j$ are very special (we described properties
of their coefficients in the paragraphs above). It seems highly likely
that such parameters $\lambda$ do not exist, but we were unable to
prove this.

\section{Appendix}
\label{sec:topological-fact}

Here we will prove the topological fact we have been often using.

\begin{lemma} \label{lm:top}
  Let $X$, $Y$ be topological spaces, and
  $X$ be compact.  Let $F: X\to Y$ be continuous.  Let $B$ be a subset
  of $Y$, $A = F^{-1}(B)$, the set $F(A)$ be connected, and $F|_A$ be
  monotone. Then $A$ is connected.
\end{lemma}

\begin{proof}
  Suppose that $A$ is not connected, so there exists a separation of
  $A$. This means that there are two non-empty subsets $A_1$ and $A_2$
  such that $A = A_1 \cup A_2$, $\bar A_1 \cap A_2 = \emptyset =  A_1
  \cap \bar A_2$. Set $B_k=F(A_k)$ where $k=1,2$.

  The sets $B_1$, $B_2$ are non-empty because $A_1$ and $A_2$ are non
  empty. Since $F(A)$ is connected, $B_1$ and $B_2$ cannot form a
  separation of $F(A)$, thus the closure of one of them should have
  non-empty intersection with the other. Assume
  $B_1\cap \bar B_2 \neq \emptyset$ and let
  $y_0 \in B_1 \cap \bar B_2$. Since $y_0 \in B_1$ there exists
  $x_0 \in A_1$ such that $F(x_0)=y_0$. Also, take a sequence of
  $y_i \in B_2$ converging to $y_0$ and let $x_i\in A_2$ be such that
  $F(x_i)=y_i$. The space $X$ is compact, so we can take a subsequence
  $x_{i_j}$ converging to some $x_\infty \in \bar A_2$.  From the
  continuity of $F$ it follows that $F(x_\infty)=y_0$.

  The map $F|_A$ is monotone, therefore $F^{-1}(y_0)$ is connected.
  We know that $x_0 \in A_1$, hence $F^{-1}(y_0) \subset A_1$.  On the
  other hand $x_\infty$ belongs to both $\bar A_2$ and $F^{-1}(y_0)$,
  so the intersection of $A_1$ and $\bar A_2$ are non-empty. This is a
  contradiction.
\end{proof}

\bibliographystyle{alpha}
\bibliography{../tesis.bib,../mendeley/library}

\end{document}